\documentclass[a4paper,11pt,reqno,noindent]{amsart}
\usepackage[centertags]{amsmath}
\usepackage{amsfonts,amssymb,amsthm,dsfont,cases,amscd,esint,enumerate, stmaryrd}
\usepackage[T1]{fontenc}
\usepackage[english]{babel}
\usepackage[applemac]{inputenc}
\usepackage{newlfont}
\usepackage{color}
\usepackage[body={17cm,21.5cm}, top = 1in, bottom = 1.5in, centering]{geometry} 
\usepackage{fancyhdr}
\usepackage{enumerate}
\usepackage{comment}

\theoremstyle{plain}
\newtheorem{theor}{Theorem}[section]
\newtheorem{lem}[theor]{Lemma}
\newtheorem{prop}[theor]{Proposition}

\theoremstyle{definition}

\newtheorem{rem}[theor]{Remark}
\newtheorem{defin0}[theor]{Definition}

\mathchardef\emptyset="001F
\numberwithin{equation}{section}

\newcommand{\Xd}{\mathbb X}

\newcommand{\R}{\mathbb R}

\newcommand{\T}{\mathbb T}

\newcommand{\Pc}{\mathcal{P}}

\newcommand{\calP}{\mathcal P}
\newcommand{\Id}{\operatorname{Id}}
\newcommand{\E}{\mathbb{E}}
\newcommand{\Ld}{\operatorname{L}}
\newcommand{\Div}{{\operatorname{div}}}
\newcommand{\Sym}{{\operatorname{sym}}}
\newcommand{\step}[1]{\noindent \textit{Step} #1.}

\newcommand{\expecM}[1]{\mathbb{E}\bigg[ #1 \bigg]}
\newcommand{\ddr}{\mathrm{d}}

\usepackage[colorlinks,citecolor=black,urlcolor=black]{hyperref}
\usepackage{tikz}
\usetikzlibrary{fit}
\usetikzlibrary{shapes.geometric}
\usetikzlibrary{calc}

\usepackage[colorlinks,citecolor=black,urlcolor=black]{hyperref}

\title{Creation of chaos for interacting Brownian particles}
\author[A. Bernou]{Armand Bernou}
\address[Armand Bernou]{Universit\'e de Lyon, Universit\'e Claude Bernard Lyon 1, Laboratoire de Sciences Actuarielle et Financi\`ere, 50 Avenue Tony Garnier, F-69007 Lyon, France}
\email{armand.bernou@univ-lyon1.fr}
\author[M. Duerinckx]{Mitia Duerinckx}
\address[Mitia Duerinckx]{Universit\'e Libre de Bruxelles, D\'epartement de Math\'ematique, 1050~Brussels, Belgium}
\email{mitia.duerinckx@ulb.be}
\author[M. M\'enard]{Matthieu M\'enard}
\address[Matthieu M\'enard]{Universit\'e Libre de Bruxelles, D\'epartement de Math\'ematique, 1050~Brussels, Belgium}
\email{matthieu.menard@ulb.be}

\begin{document}

\begin{abstract}
We consider a system of $N$ Brownian particles, with or without inertia, interacting in the mean-field regime via a weak, smooth, long-range potential, and starting initially from an arbitrary exchangeable $N$-particle distribution. In this model framework, we establish a fine version of the so-called creation-of-chaos phenomenon: in weak norms, the mean-field approximation for a typical particle is shown to hold with an accuracy $O(N^{-1})$ up to an error due solely to initial pair correlations, which is damped exponentially over time. Corresponding higher-order results are also derived in the form of higher-order correlation estimates.
The approach is new and easily adaptable: we start from suboptimal correlation estimates obtained from an elementary use of It\^o's calculus on moments of the empirical measure, together with ergodic properties of the mean-field dynamics, and these bounds are then made optimal after combination with PDE estimates on the BBKY hierarchy.

\bigskip\noindent
{\sc MSC-class:} Primary 60K35; Secondary 35Q70, 82C22, 82C31, 35Q84.
\end{abstract}

\maketitle

\bigskip

\setcounter{tocdepth}{1}
\tableofcontents
\allowdisplaybreaks

\section{Introduction}
\subsection{General overview}
We consider the Langevin dynamics for a system of $N$ Brownian particles with mean-field interactions, moving in a confining potential in $\R^d$, $d \ge 1$, as described by the following system of coupled SDEs: for $1 \le i \le N,$
\begin{align}\label{eq:Langevin}
\left\{\begin{array}{ll}
\ddr X^{i,N}_t=V^{i,N}_t\ddr t,&\\[2mm]
\ddr V^{i,N}_t =-\tfrac\kappa N\sum_{j=1}^N\nabla W(X^{i,N}_t-X_t^{j,N})\,\ddr t-\tfrac\beta2 V^{i,N}_t\ddr t - \nabla A(X^{i,N}_t)\,\ddr t+\ddr B_t^{i}, &\quad t \ge 0,
\end{array}\right.
\end{align}
where $\{Z^{i,N}:=(X^{i,N},V^{i,N})\}_{1\le i\le N}$ is the set of positions and velocities of the particles in the phase space $\Xd:=\R^d\times\R^d$, where $W:\R^d\to\R$ is a long-range interaction potential with the action-reaction condition $W(x)=W(-x)$, where $A$ is a uniformly convex confining potential, where~$\{B^i\}_i$ are i.i.d.\@ $d$-dimensional Brownian motions, and where $\kappa,\beta>0$ are given constants.
For simplicity, we choose the confining potential $A$ to be quadratic,
\begin{equation}\label{eq:quadr-A}
A(x)\,:=\,\tfrac12a|x|^2,\qquad\text{for some $a>0$},
\end{equation}
although perturbations can be considered as well (see Section~\ref{subsec:generalizations}).
Turning to a statistical description of the system, in terms of a probability density $F^N$ on the $N$-particle phase space $\Xd^N$, the Langevin dynamics~\eqref{eq:Langevin} is equivalent to the corresponding Liouville equation
\begin{multline}\label{eq:Liouville}
\partial_tF^N+\sum_{i=1}^Nv_i\cdot\nabla_{x_i}F^N\,=\,\tfrac12\sum_{i=1}^N\Div_{v_i}((\nabla_{v_i}+\beta v_i)F^N)\\[-3mm]
+\tfrac\kappa N\sum_{i,j=1}^N\nabla W(x_i-x_j)\cdot\nabla_{v_i}F^N+\sum_{i=1}^N\nabla A(x_i)\cdot\nabla_{v_i}F^N,\qquad t\ge0.
\end{multline}
In this contribution, we investigate the behavior of the system in the regime of a large number $N\gg1$ of particles in the general setting of initial data $\{Z^{i,N}|_{t=0}\}_{1\le i\le N}$ that are only assumed to be exchangeable (that is, their law $F^N|_{t=0}$ is invariant under permutation of the particles), but that may be strongly correlated.
Let us recall two standard types of results for the system:
\begin{enumerate}[---]
\item {\it Propagation of chaos:}\\
If particles have weak correlations initially, then propagation of chaos is expected to hold, which means that only little correlations can build over time. As a consequence, the distribution of a typical particle is approximated by the solution $\mu$ of the mean-field Vlasov--Fokker--Planck equation
\begin{equation}\label{eq:VFP}
\left\{\begin{array}{l}
\partial_t\mu+v\cdot\nabla_x\mu=\tfrac12\Div_v((\nabla_v+\beta v)\mu)+ (\nabla A + \kappa\nabla W\ast \mu)\cdot \nabla_v\mu,\qquad t\ge0,\\[1mm]
\mu|_{t=0}=\mu_\circ,
\end{array}\right.
\end{equation}
where $\mu_\circ$ stands for the initial single-particle distribution. Such a result has been established in various settings and we refer e.g.\@ to the recent work~\cite{BJS-22,BDJ-24} in case of singular interactions.
More precisely, for all $1\le m\le N$, denoting by~$F^{N,m}$ the distribution of $m$ typical particles in the system, as given by the $m$-th marginal of $F^N$,
\[F^{N,m}(z_1,\ldots,z_m) = \int_{\Xd^{N-m}} F^N(z_1,\ldots,z_N) \, \ddr z_{m+1} \ldots \ddr z_N, \qquad 1 \le m \le N,\]
we find that $F^{N,m}$ converges to $\mu^{\otimes m}$ as $N\uparrow\infty$ for any fixed $m$.
Beyond this qualitative convergence, an optimal $O(N^{-1})$ error estimate for the convergence of marginals was first obtained in~\cite{MD-21} for smooth interactions, and later in~\cite{Lacker-21} for bounded non-smooth interactions in a better metric: if correlations vanish initially, then there holds
\begin{align} 
\label{eq:prop-chaos} 
\|F^{N,m}_t-\mu^{\otimes m}_t\|_{TV}\,\le\,Ce^{Ct}mN^{-1},\qquad\text{for all $1\le m\le N$ and $t\ge0$}.
\end{align}
If in addition the intensity $\kappa$ of interaction forces is small enough, which ensures the uniqueness of the mean-field equilibrium, then this error estimate~\eqref{eq:prop-chaos} holds uniformly in time without the exponential factor~\cite{Lacker_LeFlem_2022,Blaustein-Jabin-Soler-24} (see also previous work using coupling methods but only producing a suboptimal convergence rate~\cite{Bolley_2010,Guillin2022}).
\smallskip\item {\it Relaxation to equilibrium:}\\
If interaction forces are weak enough, even if particles are strongly correlated initially, their distribution is known to relax to Gibbs equilibrium on long times.
For Lipschitz forces, the following uniform-in-$N$ exponential estimate was established in~\cite{Bolley_2010},
\[\mathcal W_2(F_t^{N,m},M^{N,m})\,\le\,C_me^{-\frac1Ct},\qquad\text{for all $1\le m\le N$ and $t\ge0$},\]
where $M^{N,m}$ stands for the $m$-th marginal of the $N$-particle Gibbs equilibrium $M^N$.
As the latter exhibits weak correlations, this entails in particular that initial correlations between the particles must be exponentially damped over time.
\end{enumerate}
Both uniform-in-time propagation of chaos and convergence to Gibbs equilibrium rely essentially on the same phenomenon: due to the diffusion, correlations cannot accumulate too much over time and must somehow be damped.
This leads to the natural related question of {\it creation of chaos} as first raised by Del Moral and Tugaut~\cite{DelMoral_2019} (see also ``generation of chaos'' in~\cite{Lukkarinen-24,Rosenzweig-Serfaty-23}): {\it for general exchangeable initial data, is it possible to show that the mean-field approximation holds to order $O(N^{-1})$ up to an error coming from initial correlations, which is damped exponentially over time?}

Such a result can in fact be deduced immediately from the coupling argument of~\cite{Bolley_2010}, but only with a suboptimal convergence rate $O(N^{-1/2})$: more precisely, for weak Lipschitz forces,
it follows from~\cite[Theorem~2 and Remark~12]{Bolley_2010} that
\begin{align}\label{eq:creation_strong}
\mathcal{W}_2(F^{N,m}_t, \mu_t^{\otimes m}) \le C_m \, N^{-\frac12} +  C_m \, e^{-\frac1Ct} \, \mathcal{W}_2(F^{N,m}_0, \mu_\circ^{\otimes m}), \quad \hbox{ for all } 1 \le m \le N \hbox{ and } t \ge 0.
\end{align}
Our first objective is to achieve the optimal rate $O(N^{-1})$ in this creation-of-chaos result, which to our knowledge is not possible via the same coupling methods. More precisely, we look for bounds of the following form, in some suitable metric,
\begin{equation}\label{eq:creation0}
\|F_t^{N,m}-\mu_t^{\otimes m}\|\,\le\,C_mN^{-1}+C_me^{-\frac1Ct}\|F_\circ^{N,m}-\mu_\circ^{\otimes m}\|,\quad\text{for all $1\le m\le N$ and $t\ge0$}.
\end{equation}
In the sequel, we establish estimates of this kind for weak smooth interactions, in negative Sobolev norms instead of Wasserstein distance; see Theorem~\ref{th:creation} below.

While virtually any proof of uniform-in-time propagation of chaos comes with some associated creation-of-chaos result showing the damping of initial correlations, we emphasize that clean estimates of the form~\eqref{eq:creation0} are more difficult to establish. In particular, there has been a lot of recent progress based on hierarchical approaches to uniform-in-time propagation of chaos~\cite{Lacker-21,BJS-22,Lacker_LeFlem_2022,Blaustein-Jabin-Soler-24,BDJ-24},
starting from the BBGKY hierarchy of equations for marginals,
but such methods require to deal with information on the {\it whole} hierarchy {$\{F^{N,m}-\mu^{\otimes m}\}_{1\le m\le N}$} at once and can thus a priori not be used to obtain estimates of the form~\eqref{eq:creation0} for any {\it fixed} $m$. For overdamped systems (see Section~\ref{subsec:overdamped} below), modulated energy methods can also be used to establish creation of chaos, see~\cite{Rosenzweig-Serfaty-23}, but they require much stronger initial information in the form of the relative entropy of the~$N$-particle distribution with respect to a tensorized state, which again cannot lead to clean estimates of the form~\eqref{eq:creation0} for marginals.

Next to this question on the creation and propagation of chaos, we further refine our analysis to corresponding correlation estimates.
Many-particle correlation functions provide finer information on the propagation of chaos in the system and we shall see that creation of chaos holds more generally at the level of arbitrary high-order correlations.
Recall that the two-particle correlation function is defined as
\[G^{N,2}\,:=\,F^{N,2}-(F^{N,1})^{\otimes2},\]
which captures the defect to propagation of chaos~\eqref{eq:prop-chaos} at the level of two-particle statistics.
This quantity alone does not allow to reconstruct the full particle density $F^N$:
we define higher-order correlation functions $\{G^{N,m}\}_{2\le m\le N}$ as suitable polynomial combinations of marginals of~$F^N$ in such a way that the full particle distribution $F^N$ is recovered in form of a cluster expansion,
\begin{equation}\label{eq:cluster-exp0}
F^N(z_1,\ldots,z_N)\,=\,\sum_{\pi\vdash \llbracket N\rrbracket}\prod_{B\in\pi}G^{N,\sharp B}(z_B),
\end{equation}
where $\pi$ runs through the list of all partitions of the index set $\llbracket N\rrbracket:=\{1,\ldots,N\}$, where $B$ runs through the list of blocks of the partition $\pi$, where $\sharp B$ is the cardinality of $B$, where for \mbox{$B=\{i_1,\ldots, i_l\}\subset\llbracket N\rrbracket$} we write $z_B=(z_{i_1},\ldots,z_{i_l})$, and where we have set $G^{N,1}:=F^{N,1}$. As is easily checked, correlation functions are fully determined by prescribing~\eqref{eq:cluster-exp0} together with the requirement $\int_\Xd G^{N,m}(z_1,\ldots,z_m)\,\ddr z_l=0$ for $2 \le m \le N$ and $1\le l\le m$.
If initial correlations vanish, standard propagation of chaos leads to $G^{N,2}_t=O(N^{-1})$,
and a formal analysis of the BBGKY hierarchy further leads to expect
\begin{equation}\label{eq:refined-prop-chaos}
G_t^{N,m}\,=\,O(N^{1-m}),\qquad2\le m\le N.
\end{equation}
This is referred to as {\it higher-order propagation of chaos} and is a key tool to describe deviations from mean-field theory, see e.g.~\cite{DSR-21,MD-21,DJ-24} and in particular~\cite[Corollary 1.3]{BD_2024} for the derivation of the so-called Bogolyubov correction to mean field. Such estimates have been obtained in several settings with an exponential time growth: non-Brownian particle systems with smooth interactions were first addressed in~\cite{MD-21}, while in~\cite{Hess_Childs_2023} the overdamped dynamics
was further covered for bounded non-smooth interactions. If interactions are weak enough, we proved in our previous work~\cite{BD_2024} that~\eqref{eq:refined-prop-chaos} actually holds uniformly in time, and the same was obtained later in~\cite{Xie-24} for the overdamped dynamics with bounded non-smooth interactions. All these results have been obtained only on the condition that initial correlations vanish.
Yet, uniform-in-time versions of~\eqref{eq:refined-prop-chaos} hint towards a corresponding `higher-order' creation-of-chaos phenomenon at the level of correlation functions:
{\it for general exchangeable initial data, is it possible to show that the $m$-particle correlation $G^{N,m}$ is of order $O(N^{1-m})$ only up to an error coming from initial $m$-particle correlations (or at least, from a finite number of high-order correlations), which is damped exponentially over time?}
For two-particle correlations, it already follows from~\eqref{eq:creation0} that
\begin{align}\label{eq:prelim_estim_G^2}
\|G^{N,2}_t\| \,\le\, C N^{-1} + C e^{-\frac1Ct}\|G^{N,2}_\circ\|.
\end{align}
In the sequel, we obtain similar estimates for all higher-order correlations, see Theorem~\ref{th:estim-GNm} below. The only other result of this kind that we are aware of is the recent work of Lukkarinen and Vuoksenmaa~\cite{Lukkarinen-24} regarding energy cumulants for the stochastic Kac model.
In that work, however, the authors appeal to hierarchical techniques, which require to deal with the whole hierarchy $\{G^{N,m}\}_{1\le m\le N}$ at once and do not yield simpler estimates of the form~\eqref{eq:prelim_estim_G^2} for correlations of a fixed order.

\subsection*{Notation}$ $
\begin{enumerate}[---]
\item We denote by $C\ge1$ any constant that only depends on the space dimension $d$, or possibly on other controlled quantities when specified. We use the notation~$\lesssim$ for~$\le C\times$ up to such a multiplicative constant~$C$. We write~$a\simeq b$ when both $a\lesssim b$ and $b\lesssim a$ hold. We also use the notation $\ll$ (resp.~$\gg$) for $\le \frac1C\times$ (resp.~$\ge C\times$) for some large enough constant $C$. We occasionally add subscripts to $C,\lesssim,\simeq,\ll,\gg$ to indicate dependence on other parameters.
\smallskip\item For any two integers $m\ge n\ge0$, we use the short-hand notation $\llbracket m,n \rrbracket:=\{m, m+1,\dots, n\}$, and in addition for any integer $m\ge1$ we set $\llbracket m \rrbracket := \llbracket 1, m \rrbracket$.
\smallskip\item For all $z \in \Xd$, we use the notation $\langle z \rangle := (1 + |z|^2)^{\frac12}$.
\smallskip\item For any measurable space $(E,\mathcal{E})$, we write $\mathcal{P}(E)$ for the space of probability measures on $E$. We endow $\Xd$ with the Borel $\sigma$-algebra.
\smallskip\item For a measure $\mu \in \calP(\Xd)$, we use the following short-hand notation for its moments,
\begin{equation}\label{eq:def-mom}
Q_r(\mu)\,:=\,\int_\Xd\langle z\rangle^r\mu(\ddr z),\qquad r\ge0.
\end{equation}
\end{enumerate}

\subsection{Main results}
We start by defining the appropriate functional framework for our results:
for~\mbox{$r \geq 0$,} $1 < q < 2$, and $0 < p \leq 1$, we consider the following weighted negative Sobolev norms on functions~$h:\Xd\to\R$,
\begin{equation}\label{eq:definition_W_tilde_1}
\|h\|_{W_p^{-r,q}(\Xd)} \,:=\, \sup \bigg\{\Big|\int_{\Xd}\varphi h \Big|~:~\varphi \in C_c^\infty(\Xd),~~\|\langle \cdot \rangle^{-p}\varphi\|_{W^{r,q'}(\Xd)} \le 1 \bigg\}
\end{equation}
and for $m\ge1$ and functions~$h:\Xd^m\to\R$,
\begin{equation}\label{eq:definition_W_tilde}
\|h\|_{W_p^{-r,q}(\Xd)^{\otimes m}} \,:=\, \sup \bigg\{\Big|\int_{\Xd^m}\varphi^{\otimes m} h \Big|~:~\varphi \in C_c^\infty(\Xd),~~\|\langle \cdot \rangle^{-p}\varphi\|_{W^{r,q'}(\Xd)} \le 1 \bigg\}.\footnote{This is the norm of the injective tensor product $W_p^{-r,q}(\Xd)^{\otimes m}$, which of course differs from the norm of $W^{-r,q}_p(\Xd^m)$.}
\end{equation}
The choice of these spaces is dictated by the available ergodic estimates on the linearized mean-field evolution, which are not known to hold on the simpler unweighted spaces $W^{-k,1}(\Xd)$; see e.g.~\cite{BD_2024} and Lemma~\ref{lem:ergodic} below.

Our first main result justifies the desired estimates~\eqref{eq:creation0} on the creation and propagation of chaos in the above norms. The estimates appear essentially optimal, except for the dependence on $m$, for the strong regularity assumption, and perhaps for the loss of regularity in the norms.

\begin{theor}[Creation of chaos]\label{th:creation}
Let the interaction potential $W$ be even, smooth and decaying in the sense of $W\in W^{d+3,\infty}\cap H^s(\R^d)$ for some $s>\frac d2+5$, and let the confining potential $A$ be quadratic~\eqref{eq:quadr-A}.
Given an exchangeable $N$-particle distribution $F^N_\circ\in\Pc(\Xd^N)$ with all bounded moments, consider the weak solution $F^N\in C(\R^+;\Pc(\Xd^N))$ of the Liouville equation~\eqref{eq:Liouville} with initial data $F^N_\circ$. Given $\mu_\circ\in\Pc(\Xd)$ with bounded second moments, also consider the weak solution $\mu\in C(\R^+;\Pc(\Xd))$ of the mean-field equation~\eqref{eq:VFP} with initial data $\mu_\circ$.
Then there is $\lambda_0>0$ such that the following holds:
given $1<q\le2$ and $0<p\le\frac16$ with $pq'\gg1$ large enough, and provided that $0<\kappa\ll1$ is small enough, we have for all $1\le m\le N$ and~$t\ge0$,\footnote{In this statement, the exponent $\lambda_0$ only depends on~$d,\beta,a$, $\|W\|_{W^{1,\infty}(\R^d)}$; the condition $pq'\gg1$ only depends on~$d,\beta,a$; the smallness condition $\kappa\ll1$ further depends on $p,q,$ $Q_2(\mu_\circ)$, $Q_2(F^{N,1}_\circ)$, and on controlled norms of $W$; and the multiplicative constant in~\eqref{eq:res-creation} further depends on $m$ and~$Q_{mp}(F^{N,1}_\circ)$.}
\begin{equation}\label{eq:res-creation}
\|F_t^{N,m}-\mu_t^{\otimes m}\|_{W_p^{-3,q}(\Xd)^{\otimes m}}
\,\lesssim\, N^{-1}
+\sum_{j=1}^{\tilde m} e^{-j\lambda_0pt}\|F_\circ^{N,j}-\mu_\circ^{\otimes j}\|_{W^{-2,q}_{3p}(\Xd)^{\otimes j}},
\end{equation}
where we have set for abbreviation $\tilde m=m$ for $m$ even and $\tilde m=m+1$ for $m$ odd.
\end{theor}

We turn to the creation of chaos at the level of correlation estimates. The precise combination of initial correlations that needs to be tracked quickly becomes quite cumbersome for high-order correlations. Although its expression is easily inferred from the proof, we stick here, for readability, to a simpler statement where we assume that correlations have a definite scaling initially.
We emphasize that for estimates on $m$-particle correlations we only need information on initial $j$-particle correlations with $j\le2m-1$, cf.~\eqref{eq:hypothesis_chaoticity_initial_datum}, in contrast to the much stronger requirements for hierarchical arguments.
In the case of chaotic initial data, we recover in particular the uniform-in-time correlation estimates obtained in~\cite{BD_2024}, cf.~\eqref{eq:correl-iid}.

\begin{theor}[Creation of chaos in correlation estimates]\label{th:estim-GNm}
Let $W, A,F^N,\lambda_0,p,q,\kappa$ be as in the statement of Theorem~\ref{th:creation}, and further assume $W\in C^\infty_b\cap H^\infty(\R^d)$.
Then the two-particle correlation function satisfies for all $t\ge0$,\footnote{The multiplicative constant in~\eqref{eq:estim-GN2-cr} only depends on $d, \beta, a, Q_{2}(F^{N,1}_\circ)$, and on $\|W\|_{W^{d+3,\infty}\cap H^s(\Xd)}$ for some $s>\frac d2+5$; the additional smallness condition on $\kappa$ only depends on $p,q,Q_2(F_\circ^{N,1}),m$; and the multiplicative constants in~\eqref{eq:conclusion_thm_estim_GNm} and~\eqref{eq:correl-iid} further depend on $Q_{mp}(F^{N,1}_\circ)$, $\|W\|_{W^{d+m+2,\infty}(\R^d)}$, and on the constants $(C_j^\circ)_{2 \leq j \leq 2m-1}$ in~\eqref{eq:hypothesis_chaoticity_initial_datum}.}
\begin{equation}\label{eq:estim-GN2-cr}
\|G^{N,2}_t\|_{W^{-3,q}_p(\Xd)^{\otimes2}} \,\lesssim\, N^{-1} + e^{-\lambda_0pt}\|G^{N,2}_\circ\|_{W^{-2,q}_{3p}(\Xd)^{\otimes2}}.
\end{equation}
In addition, given $2\le m\le N$ and $\alpha \in [0,1]$, if we have initially
\begin{equation}\label{eq:hypothesis_chaoticity_initial_datum}
\|G_\circ^{N,j}\|_{W^{-2,q}_{3p}(\Xd)^{\otimes j}}\,\le\,C_j^\circ N^{\alpha(1-j)},\qquad\text{for all $2\le j\le 2m - 1$},
\end{equation}
and if $\kappa$ is small enough (further depending on $m$), then we have for all~$t \ge 0$,
\begin{equation}
\label{eq:conclusion_thm_estim_GNm} 
\|G^{N,m}_t\|_{W^{-r_m,q_m}_{p_m}(\Xd)^{\otimes m}} \,\lesssim\,N^{1-m}+e^{-\lambda_0 p_mt} N^{\alpha(1-m)},
\end{equation}
for $r_m=m+1$, some $q_m'\simeq_mq'$, and $p_m\simeq_m p$ with $p_mq_m'=pq'$.
In particular, if the initial data are chaotic, that is, if $F_\circ^N=\mu_\circ^{\otimes N}$ for some $\mu_\circ\in\Pc(\Xd)$, then we have for all $2\le m\le N$ and $t\ge0$,
\begin{equation}\label{eq:correl-iid}
\| G^{N,m}_t\|_{W^{-r_m,q_m}_{p_m}(\Xd)^{\otimes m}}\,\lesssim\,N^{1-m}.
\end{equation}
\end{theor}

\begin{rem}[Smallness assumption]\label{rem:sub_optimal_creation_chaos}
In the above results~\eqref{eq:conclusion_thm_estim_GNm}--\eqref{eq:correl-iid}, the smallness assumption on $\kappa$ depends on the order $m$ of the correlation function of interest. This is necessary in our approach (except in the simpler setting of the overdamped dynamics, see Section~\ref{subsec:overdamped} below), but it can be at least partially relaxed.
More precisely, focusing for shortness on the particular case of chaotic initial data, $F^N_\circ=\mu_\circ^{\otimes N}$, we can prove the following suboptimal bounds: letting $W, A, F^N, p, q, \kappa$ be as in Theorem~\ref{th:creation} (hence, with $\kappa$ independent of $m$), we have for all $2 \le m \le N$ and~$t \ge 0$,
\begin{align*}
\|G^{N,m}_t\|_{W^{-2,q}_{p_m}(\Xd)^{\otimes m}} \lesssim N^{-\frac{m}{2}}. 
\end{align*}
(To see this, we refer to Step~1 of the proof of Theorem~\ref{th:estim-GNm} in Section~\ref{sec:correl}, simply specializing it to the case of chaotic data, hence $\alpha=1$.)
\end{rem}

\subsection{Extensions}
\label{subsec:generalizations}
We briefly discuss some simple extensions of the above results. The adaptation to the Kac model, based on ergodic estimates for the space-homogeneous Boltzmann equation and the use of a similar BBGKY hierarchy, is a bit more demanding and is postponed to a separate work, improving the recent result of Lukkarinen and Vuoksenmaa in~\cite{Lukkarinen-24}. 

\subsubsection{Overdamped dynamics}\label{subsec:overdamped}
Our results also apply to the overdamped Langevin dynamics, as described by the following system of coupled SDEs: for $1\le i\le N$,
\begin{equation}\label{eq:Brownian_system}
dY^{i,N}_t = \tfrac{\kappa}{N} \sum_{j=1}^N K(Y^{i,N}_t-Y^{j,N}_t)\, \ddr t - \nabla A(Y^{i,N}_t) + dB^i_t, \qquad t \ge 0,
\end{equation}
where $K:\R^d \to \R^d$ is a smooth force kernel with the action-reaction condition $K(x)=-K(-x)$ (we may take $K$ of the form $K=-\nabla W$, but this is not necessary).
In terms of a probability density $\tilde F^N$ on the $N$-particle phase space $(\R^d)^N$, this particle dynamics is equivalent to the following Liouville equation,
\begin{equation}\label{eq:Liouville_Brownian}
\partial_t \tilde F^N = \tfrac12 \sum_{i=1}^N \triangle_{x_i} \tilde F^N_t-\frac{\kappa}{N} \sum_{i, j = 1}^N \Div_{x_i} \big(K(x_i-x_j)\tilde F^N\big)+ \sum_{i=1}^N \Div_{x_i} \big(\nabla A(x_i) \tilde F^N\big).
\end{equation}
The associated mean-field equation is now the McKean--Vlasov equation
\begin{equation}\label{eq:MF_Brownian}
\partial_t \tilde \mu = \tfrac12 \triangle \tilde \mu - \kappa\, \Div ( \tilde \mu (K \ast \tilde \mu) )+ \Div ( \tilde \mu \nabla A ).
\end{equation}
In this setting, by standard parabolic theory, ergodic estimates for the linearized mean-field equation hold in the {\it unweighted} spaces $W^{-k,1}(\R^d)$ for $k\ge0$, see~\cite{Delarue_Tse_21,BD_2024}. This simplifies our proof substantially: while working in {\it weighted} Sobolev spaces in the kinetic setting forces us to use fine concentration estimates for moments of the particle dynamics, this is no longer needed here. In addition, when adapting the proof to the present unweighted framework, we find that the smallness requirement on~$\kappa$ in correlation estimates is no longer required to depend on the order of the correlation.
More precisely, we obtain the following.

\begin{theor}[Creation of chaos --- overdamped setting]
\label{th:brownian}
Let the interaction force $K$ be odd, smooth and decaying in the sense of $K \in C^\infty_b\cap H^\infty(\R^d)$, and let the confining potential $A$ be quadratic~\eqref{eq:quadr-A}.
Given an exchangeable $N$-particle distribution $\tilde F^N_\circ \in \calP((\R^d)^N)$ with bounded second moments, we consider the solution $\tilde F^N \in C(\R^+; \mathcal{P}((\R^d)^N))$ of~\eqref{eq:Liouville_Brownian} with initial data $\tilde F^N_\circ$ and denote by $\{\tilde G^{N,m}\}_{2\le m \le N}$ the corresponding correlation functions. Given $\tilde \mu_\circ \in \mathcal{P}(\R^d)$ with bounded second moments, also consider the weak solution $\tilde \mu \in C(\R^+; \calP (\R^d))$ of the McKean--Vlasov equation~\eqref{eq:MF_Brownian} with initial data~$\tilde \mu_\circ$.
Then there is $\lambda_0 > 0$ such that the following hold for $0<\kappa\ll1$ small enough.\footnote{In this statement, the exponent $\lambda_0$ only depends on $d,\beta,a$, $\|W\|_{W^{1,\infty}(\R^d)}$; the smallness condition on $\kappa$ further depends on $Q_2(\mu_\circ)$, $Q_2(F_\circ^{N,1})$, and on some norms of $W$; the multiplicative constant in~\eqref{eq:estim-FNm-mum-tilde} further depends on $m$ and on some norms of $W$ also depending on $m$; the multiplicative constant in~\eqref{eq:ccl_many_particle_brownian} further depends on the constants $(C_j^\circ)_{2\le j\le 2m-1}$ in~\eqref{eq:init_scaling_brownian_2}.}
\begin{enumerate}[---]
\item For all $1\le m\le N$ and $t\ge0$ we have
\begin{equation}\label{eq:estim-FNm-mum-tilde}
\|\tilde F^{N,m}_t - \tilde \mu_t^{\otimes m}\|_{W^{-3,1}(\R^d)^{\otimes m}} \,\lesssim\, N^{-1} + \sum_{j=1}^{\tilde m} e^{-j \lambda_0 t} \|\tilde F^{N,j}_\circ - \tilde \mu_\circ^{\otimes j} \|_{W^{-2, 1}(\R^d)^{\otimes j}},
\end{equation}
where we recall the short-hand notation $\tilde m=m+\mathds1_{\text{$m$ odd}}$.
\smallskip\item
Given $2\le m\le N$ and $\alpha \in [0,1]$, if we have initially
\begin{equation}\label{eq:init_scaling_brownian_2}
\|\tilde G^{N,j}_\circ\|_{W^{-2,1}(\R^d)^{\otimes j}} \,\le\,C_j^\circ {N^{\alpha(1-j)}}, \qquad \text{for all $1 \le j \le 2m-1$},
\end{equation}
then we have for all $t \ge 0$,
\begin{equation}\label{eq:ccl_many_particle_brownian}
\|\tilde G^{N,m}_t\|_{W^{-m-1,1}(\R^d)^{\otimes m}} \,\lesssim\, N^{1-m} + e^{-\lambda_0 t} N^{\alpha(1-m)}.
\end{equation}
In addition, if the initial data are chaotic, that is, if $\tilde F_\circ^N=\tilde \mu_\circ^{\otimes N}$, then for all $2\le m\le N$ and $t\ge0$,
\[\|\tilde G^{N,m}_t\|_{W^{-m-1,1}(\R^d)^{\otimes m}}\,\lesssim\,N^{1-m}.\]
\end{enumerate}
\end{theor}

\subsubsection{Periodic setting}
Our analysis is easily adapted to particle systems on the torus $\T^d$, for instance in the spatially homogeneous setting $A \equiv 0$. However, in that case, for the Langevin dynamics~\eqref{eq:Langevin}, ergodic estimates for the linearized mean-field equation only yield a super-polynomial decay $t^{-\infty}$ instead of exponential; see indeed~\cite[Theorem~56]{Villani_2009} and~\cite[Remark~2.14]{BD_2024}. As a result, Theorems~\ref{th:creation} and~\ref{th:estim-GNm} then take the same form as above up to replacing the exponential time decay~$e^{-ct}$ by~$t^{-\infty}$.
For the overdamped dynamics~\eqref{eq:Brownian_system}, on the other hand, the exponential decay still holds in the periodic setting.

\subsubsection{Non-quadratic confinement}
Both for the Langevin dynamics~\eqref{eq:Langevin} and the overdamped dynamics~\eqref{eq:Brownian_system}, our results would hold in the same form if the quadratic
confinement~\eqref{eq:quadr-A} were replaced by $A(x) = a|x|^2 + A'(x)$ for some $a > 0$ and some smooth potential $A' \in C^{\infty}_c(\R^d)$, provided that $\|\nabla^2 A'\|_{L^\infty(\R^d)}$ is small enough (only depending on $\beta, W, a)$. In that case, we can still appeal to~\cite{BD_2024} to ensure the validity of the ergodic estimates of Lemma~\ref{lem:ergodic} below, while the rest of our approach can be adapted directly without major difficulties.

\subsubsection{More general interactions}
While we focus for shortness on pairwise interactions deriving from a potential, the term $\frac\kappa{N} \sum_{j=1}^N \nabla W(X^{i,N}-X^{j,N})$ in~\eqref{eq:Langevin} (or correspondingly in~\eqref{eq:Brownian_system})
could be replaced by
a more general interaction term of the form $\kappa\,b(X^{i,N}, \mu^N)$ for some smooth functional $b: \R^d \times \mathcal{P}(\R^d) \to \R$,
where $\mu^N$ stands for the empirical measure
\begin{equation}\label{eq:def_muNt}
\mu^N_t \,:=\, \frac1{N}\sum_{i=1}^N \delta_{Z^{i,N}_t} \in \calP(\Xd).
\end{equation}
Equation~\eqref{eq:Langevin} corresponds to the particular choice $b(x,\mu) = -\nabla W \ast \mu(x)$, but our analysis is easily adapted to more general functionals $b$ under suitable regularity assumptions both on $b$ and on its derivatives with respect to the measure argument, together with a corresponding smallness condition on $\kappa$. We skip the details and refer to~\cite{Delarue_Tse_21} for the needed adaptations (see in particular the regularity assumptions (Reg) and (Lip) in~\cite[Section~2.3.2]{Delarue_Tse_21}).

\subsection{Strategy and plan of the paper}
Our starting point is the derivation of new estimates for quantities of the form
\[A^m(t) \,:=\, \sup_\varphi\mathbb{E}\bigg[\Big|\int_{\Xd} \varphi (\mu^N_t - \mu_t) \Big|^m\bigg],\]
where the supremum runs over a suitable class of smooth test functions $\varphi$. Analyzing the dynamics of those quantities and appealing to the ergodic estimates of the linearized mean-field equation, we are led to proving roughly that for all $t \ge 0$,
\begin{equation}\label{eq:prop-fund}
A^m(t) \,\le\, \Big(\frac{Cm}{N}\Big)^\frac{m}{2} + C^me^{-\lambda_0 m t} A^m(0),
\end{equation}
see Proposition~\ref{prop:ito}.
For the Langevin dynamics, as ergodic estimates only hold in weighted Sobolev spaces, we lose moments of the particle dynamics in the estimates, which we manage to control by means of suitable concentration estimates.

When reformulating~$A^m$ in terms of marginals or correlations, this bound~\eqref{eq:prop-fund} leads only to suboptimal estimates. In particular, as $t\uparrow\infty$, we only deduce a bound $O(N^{-1/2})$ instead of $O(N^{-1})$ for the error in the propagation of chaos, and a bound $O(N^{-m/2})$ instead of $O(N^{1-m})$ for the $m$-particle correlation.
In order to recover the optimal scalings, our next main ingredient is the BBGKY hierarchy: more precisely, appealing again to ergodic properties of the linearized mean-field equation, we show that the optimal scalings can be rigorously extracted from the BBGKY hierarchy, noting that the latter can be truncated using the suboptimal estimates obtained from~\eqref{eq:prop-fund}.

We emphasize that relatively few ingredients are used in this approach: the basic estimate~\eqref{eq:prop-fund} relies solely on ergodic estimates for the linearized mean-field equation, on concentration properties, and on some regularity of the interaction kernel. The strategy appears versatile and should be applicable to other mean-field systems, including the Kac model mentioned above, which is postponed to a future work.

\subsubsection*{Plan of the paper}
Section~\ref{sec:preliminary} is devoted to the statement of ergodic estimates, the proof of concentration bounds for moments of the particle dynamics, and most importantly the proof of~\eqref{eq:prop-fund}, cf.~Proposition~\ref{prop:ito}. Next, Sections~\ref{sec:proof_thm_main} and~\ref{sec:correl} are devoted to the proof of Theorems~\ref{th:creation} and~\ref{th:estim-GNm} as a consequence of~\eqref{eq:prop-fund} together with BBGKY estimates.

\section{Preliminary results}
\label{sec:preliminary}
As correlations are driven by the linearized mean-field operator, a key ingredient to understand the damping of initial correlations is naturally given by ergodic estimates for linearized mean field.
Recall that $\mu$ stands for the solution of the mean-field equation~\eqref{eq:VFP} with initial data $\mu_\circ$. This equation admits a unique steady state $M$ for $0<\kappa\ll1$ small enough, which can be characterized as the unique solution to the fixed-point Gibbs equation
\begin{equation*}
M(x,v) = c_M \exp \Big[ -\beta\Big(\tfrac12 |v|^2 + A(x) + \kappa W \ast M(x) \Big) \Big], \qquad (x,v) \in \Xd,
\end{equation*}
where $c_M$ is the normalizing constant such that $\int_{\Xd} M = 1$. This fixed-point equation has indeed a unique solution when $\kappa \beta \|W\|_{L^\infty(\R^d)} < 1$.
For $h\in C^\infty_c(\Xd)$ and $t\ge0$, we consider the (non-autonomous) backward dual linearized mean-field evolution,
\begin{equation}\label{eq:def-Ust}
\left\{\begin{array}{l}
\partial_s U_{s,t}h+v\cdot\nabla_xU_{s,t}h=\Big(-\tfrac12\triangle_v+\tfrac12\beta v\cdot\nabla_v+(\nabla A+\kappa\nabla W\ast\mu_s)\cdot\nabla_v\Big)U_{s,t}h,\qquad s\le t,\\
U_{s,t}h|_{s=t}=h.
\end{array}\right.
\end{equation}
In~\cite[Proposition~7.1]{BD_2024}, we obtained the following ergodic estimates in weighted Sobolev spaces. This is a consequence of hypocoercivity techniques together with a suitable use of the enlargement theory of Mischler and Mouhot~\cite{Mischler_2016}.
(In~\cite{BD_2024}, an estimate is obtained instead for the forward primal linearized evolution, but it is equivalent to this one by duality.)

\begin{lem}[Dual ergodic estimates, \cite{BD_2024}]\label{lem:ergodic}
There exist $\kappa_0,\lambda_0>0$ (depending on $d,\beta,a$, $\|W\|_{W^{1,\infty}(\R^d)}$) such that given $\kappa\in[0,\kappa_0]$ we have for all $r\ge0$, $1<q\le2$, $0<p\le1$ with $pq'\gg1$ large enough (only depending on $d,\beta,a$), $t\ge s\ge0$, and $h\in C^\infty_c(\Xd)$,
\begin{equation*}
\Big\|\langle z\rangle^{-p} \Big(U_{s,t}h-\int_\Xd hM\Big)\Big\|_{W^{r,q'}(\Xd)}\,\lesssim\,e^{-\lambda_0p(t-s)}\|\langle z\rangle^{-p}h\|_{W^{r,q'}(\Xd)},
\end{equation*}
where the multiplicative constant only depends on $d,\beta,p,q,r,a$, $\|W\|_{W^{r+d+1,\infty}(\R^d)}$, and $Q_2(\mu_\circ)$.
\end{lem}

We emphasize that the above ergodic estimates are not known to hold in $W^{r,\infty}(\Xd)$, but only in weighted spaces $\langle z\rangle^{p} W^{r,q'}(\Xd)$ with $pq'\gg1$.
As we work in such weighted spaces, we will need to control moments of the particle dynamics. More precisely, we will appeal to the following concentration bounds for moments of the empirical measure $\mu^N$ given in~\eqref{eq:def_muNt}.

\begin{lem}[Moment estimates]\label{lem:mom}
Let $\kappa_0 > 0$ be as in Lemma~\ref{lem:ergodic}.
There exist $\lambda_0>0$ (only depending on~$\beta,a$) and $C_0>0$ (further depending on $d,\|W\|_{W^{1,\infty}(\R^d)}$), such that for all $2\le m\le N$, $0 < p \le 1$, $L\ge C_0Q_{2p}(\mu_\circ)$, and~$t\ge0$ we have
\begin{equation}\label{eq:concentr}
\E\big[(Q_{2p}(\mu_t^N)-L)_+^m\big]\, \le \, 
\Big(\frac{C_0Lm}N\Big)^\frac{m}{2}
+C_0^me^{-\lambda_0pm t}\,\E\bigg[\Big|\int_{\Xd} \langle z\rangle^{2p}(\mu_0^N - \mu_\circ)\Big|^m\bigg].
\end{equation}
In addition, for all $r\ge0$,
\begin{equation}\label{eq:mom}
\E[Q_{r}(\mu_t^N)]\,=\,Q_r(F_t^{N,1})\,\le\,C_0(C_0r)^\frac r2+C_0^re^{-\lambda_0rt}Q_r(F^{N,1}_\circ).
\end{equation}
\end{lem}

\begin{proof}
We split the proof into two steps, starting with the proof of the concentration estimate~\eqref{eq:concentr}, before turning to the more classical proof of the moment bound~\eqref{eq:mom}.

\medskip\noindent
\step1 Proof of~\eqref{eq:concentr}.\\
Let $0<p\le1$ be fixed. In the spirit of~\cite{Bolley_2010}, we consider the random process
\begin{equation*}
G_{t}^{N}\,:=\,\frac1N\sum_{i=1}^N(R_t^{i,N})^p,\qquad R_t^{i,N}:=1+ a|X_{t}^{i,N}|^2+|V_{t}^{i,N}|^2+\eta X_{t}^{i,N}\cdot V_{t}^{i,N},
\end{equation*}
for some $\eta\in(0,2\sqrt{a})$ to be chosen later on. Note that this range of $\eta$ ensures
\begin{equation}\label{eq:GtN-equiv}
R_t^{i,N}\,\simeq_{a,\eta}\,\langle Z_t^{i,N}\rangle^2,\qquad G_t^{N}\,\simeq_{a,\eta}\,Q_{2p}(\mu_t^N).
\end{equation}
The particle dynamics \eqref{eq:Langevin} and It\^o's formula  yield
\begin{multline*}
\ddr R_t^{i,N}\,=\,
-\Big(a\eta|X^{i,N}_t|^2
+(\beta-\eta)|V_t^{i,N}|^2
+\frac{\eta\beta}2X_t^{i,N}\cdot V_t^{i,N}-1\Big)\ddr t\\
-\frac\kappa{N}\sum_{j=1}^N(2V_t^{i,N}+\eta X_t^{i,N})\cdot\nabla W(X_t^{i,N}-X_t^{j,N})\ddr t
+(2V_t^{i,N}+\eta X_t^{i,N})\cdot \ddr B^i_t.
\end{multline*}
Provided that $0<\eta\ll1$ is chosen small enough (only depending on $\beta,a$), we can find $\lambda_0>0$ only depending on~$\beta,a$ such that
\begin{equation}\label{eq:ineq-dR}
\ddr R_t^{i,N}\,\le\, - 2\lambda_0 R_t^{i,N} \ddr t + C_0 \ddr t + (2 V^{i,N}_t + \eta X^{i,N}_t)\cdot \ddr B^i_t,
\end{equation}
where henceforth $C_0$ stands for a positive constant only depending on $d,\beta,a,\|W\|_{W^{1,\infty}(\R^d)}$, the value of which may change from line to line.
By It\^o's formula, using $p(p-1) \le 0$ and $(R_t^{i,N})^{p-1}\le1$, we deduce
\begin{eqnarray*}
\ddr (R_t^{i,N})^p&\le& - 2\lambda_0 p (R_t^{i,N})^p\ddr t  + C_0 p (R_t^{i,N})^{p-1}\ddr t + \tfrac{1}{2}p(p-1)(R_t^{i,N})^{p-2}|2 V^{i,N}_t + \eta X^{i,N}_t|^2\ddr t\\
&&+\, p(R_t^{i,N})^{p-1}(2 V^{i,N}_t + \eta X^{i,N}_t)\cdot \ddr B^i_t\\[1mm]
&\le&- 2\lambda_0 p (R_t^{i,N})^p \ddr t + C_0 p\,\ddr t + p(R_t^{i,N})^{p-1}(2 V^{i,N}_t + \eta X^{i,N}_t)\cdot \ddr B^i_t,
\end{eqnarray*}
and thus, summing over $i$ and recalling the definition of $G_t^N$,
\begin{align*}
\ddr G_{t}^{N} \leq -2\lambda_0 p G_{t}^{N} \ddr t + C_0 p\,\ddr t + \frac{p}{N}\sum_{i=1}^N (R_t^{i,N})^{p-1}(2 V^{i,N}_t + \eta X^{i,N}_t)\cdot \ddr B^i_t.
\end{align*}
Further using It\^o's generalized formula for convex functions, bounding classically the expectation of the local time using the quadratic variation,
we get for all $L\ge0$ and $2\le m\le N$,
\begin{multline*}
\frac{\ddr}{\ddr t}\E\big[(G_t^{N}-L)_+^m\big]\,\le\,
\E\bigg[-2\lambda_0 p m(G_t^{N}-L)_+^{m-1}G_t^{N} + C_0pm(G_t^{N}-L)_+^{m-1} \\
\hspace{2cm}+\frac{p^2 m(m-1)}{2N^2}(G_t^{N}-L)_+^{m-2}\sum_{i=1}^N (R_t^{i,N})^{2p-2}|2 V^{i,N}_t + \eta X^{i,N}_t|^2\bigg] \\
\,\le\, \E\bigg[-2\lambda_0 p m(G_t^{N}-L)_+^{m-1}G_t^{N} + C_0pm(G_t^{N}-L)_+^{m-1} 
+\frac{C_0p m^2}N(G_t^{N}-L)_+^{m-2}G_t^{N}\bigg],
\end{multline*}
hence, decomposing $G_t^N=G_t^N-L+L$, we get
for $L\ge C_0/\lambda_0$,
\begin{align*}
\frac{\ddr}{\ddr t}\,\E\big[(G_t^{N}-L)_+^m\big]\,\le\, \E\bigg[-2\lambda_0 p m(G_t^{N}-L)_+^{m} + \frac{C_0Lpm^2}{N}(G_t^{N}-L)_+^{m-2}\bigg].
\end{align*}
By Young's inequality, this entails
\begin{align*}
\frac{\ddr}{\ddr t}\E\big[(G_t^{N}-L)_+^m\big]\, \le \, -\lambda_0 p m \E[(G_t^{N}-L)_+^m] +2\lambda_0 p\Big(\frac{C_0Lm}{\lambda_0 N}\Big)^\frac{m}{2},
\end{align*}
and thus, by Gr\"onwall's inequality,
\begin{align*}
\E\big[(G_t^{N}-L)_+^m\big]\, \le \, \Big(\frac{C_0Lm}{\lambda_0N}\Big)^\frac{m}{2}+e^{-\lambda_0 p mt}\,\E\big[(G_0^{N}-L)_+^m\big].
\end{align*}
Noting that
\begin{equation*}
\E\big[(G_0^{N}-L)_+^m\big] \,\le\,\E\bigg[\Big(C_0\int_{\Xd}\langle z\rangle^{2p} \mu_0^N - L\Big)^m_+\bigg] 
\,=\,C_0^m\,\E\bigg[\Big(\int_{\Xd}\langle z\rangle^{2p} (\mu_0^N - \mu_\circ) + Q_{2p}(\mu_\circ) - C_0^{-1}L\Big)^m_+\bigg],
\end{equation*}
and further imposing $L\ge C_0Q_{2p}(\mu_\circ)$, we conclude
\begin{equation*}
\E\big[(G_t^{N}-L)_+^m\big]\, \le \, \Big(\frac{C_0Lm}{\lambda_0N}\Big)^\frac{m}{2}+C_0^me^{-\lambda_0 p m t}\,\E\bigg[\Big|\int_{\Xd} \langle z\rangle^{2p}(\mu_0^N - \mu_\circ)\Big|^m\bigg],
\end{equation*}
and the conclusion~\eqref{eq:concentr} follows.

\medskip\noindent
\step2 Proof of~\eqref{eq:mom}.\\
Using the same notation as in Step~1, and starting from~\eqref{eq:ineq-dR},
we find by It\^o's formula, for any $r>0$,
\begin{equation}\label{eq:ito-mom0}
\frac{\ddr}{\ddr t}\E\big[(R_t^{1,N})^r\big]\,\le\,- 2\lambda_0 r \E\big[(R_t^{1,N})^r\big]+C_0 r \E\big[(R_t^{1,N})^{r-1}\big]+\tfrac{1}{2}r(r-1)\E\big[(R_t^{1,N})^{r-2}|2 V^{1,N}_t + \eta X^{1,N}_t|^2\big].
\end{equation}
In case $r\ge1$, using $|2 V^{1,N}_t + \eta X^{1,N}_t|^2\lesssim R_t^{1,N}$ and Young's inequality, we find
\begin{equation*}
\frac{\ddr}{\ddr t}\E\big[(R_t^{1,N})^r\big]\,\le\,- 2\lambda_0 r \E\big[(R_t^{1,N})^r\big]+C_0 r^2 \E\big[(R_t^{1,N})^{r-1}\big]\,\le\,-\lambda_0 r \E\big[(R_t^{1,N})^r\big]+\lambda_0 \Big(\frac{C_0r}{\lambda_0}\Big)^r.
\end{equation*}
Hence, by Gr\"onwall's inequality, recalling $r \ge 1$
\begin{equation*}
\E\big[(R_t^{1,N})^r\big]\,\le\, \Big(\frac{C_0r}{\lambda_0}\Big)^r+e^{-\lambda_0rt}\,\E\big[(R_0^{1,N})^r\big],
\end{equation*}
and the conclusion~\eqref{eq:mom} then follows (up to renaming the constants), using
\[\E\big[(R_t^{1,N})^r \big]\,\simeq_r\,\E\big[Q_{2r}(\mu_t^N)\big]\,=\,\E\Big[\int_\Xd\langle z\rangle^{2r}\mu_t^N\Big]\,=\,Q_{2r}(F_t^{N,1}).\]
In case $0<r\le1$, rather using $r(r-1)\le0$ and $(R_t^{1,N})^{r-1}\le1$ in~\eqref{eq:ito-mom0}, we find
\begin{equation*}
\frac{\ddr}{\ddr t}\E\big[(R_t^{1,N})^r\big]\,\le\,- 2\lambda_0 r \E\big[(R_t^{1,N})^r\big]+C_0 r,
\end{equation*}
and the conclusion follows similarly.
\end{proof}

With the above two preliminary lemmas at hand, we may now turn to the proof of the following moment estimates on the mean-field approximation error, which we deduce from a suitable application of It\^o's formula together with a Gr\"onwall argument.
The proof can be compared with some related computations in~\cite{Cornalba-Fischer-Raithel-2023}.

\begin{prop}\label{prop:ito}
Consider the following quantities measuring the mean-field approximation error, for any $1<q\le2$, $0<p\le1$, $m\ge1$, and~$t\ge0$,
\begin{equation}\label{eq:def-ANm}
A^{N,m}_{p,q}(t)\,:=\,\sup\bigg\{\E\bigg[\Big|\int_\Xd\varphi(\mu_t^N-\mu_t)\Big|^{m}\bigg]~:~\varphi\in C^\infty_c(\Xd),~\|\langle z\rangle^{-p}\varphi\|_{W^{2,q'}(\Xd)}\le1\bigg\}.
\end{equation}
There exists $\lambda_0>0$ (only depending on $d,\beta,a,\|W\|_{W^{1,\infty}(\R^d)}$)
such that the following holds:
given $1<q\le2$ and $0<p\le1$ with $pq'\gg1$ large enough (only depending on $d,\beta,a$), there exist $\kappa_0,C_0>0$ (further depending on $p,q$, $Q_2(\mu_\circ)$, and $\|W\|_{W^{d+3,\infty}\cap H^{s}(\R^d)}$ for $s>\frac{d}{2}+3$) such that given $\kappa\in[0,\kappa_0]$ we have for all $2\le m\le N$ and $t\ge0$,
\begin{equation*}
A_{p,q}^{N,m}(t) \,\le\,
\Big(\frac{C_0m}N\Big)^\frac m2
+C_0^m e^{-\lambda_0p m t}A_{3p,q}^{N,m}(0).
\end{equation*}
\end{prop}

\begin{proof}
Let $\lambda_0$ stand for the minimum between the corresponding exponents in Lemmas~\ref{lem:ergodic} and~\ref{lem:mom}.
Let $1<q\le2$ and $0<p\le1$ be fixed with $pq'\gg_{\beta,a}1$.
Instead of controlling directly the evolution of~$A_{p,q}^{N,m}(t)$, let us consider
\begin{equation*}
D^{N,m}_{p,q}(t)\,:=\,\sup\bigg\{e^{\gamma m(t-s)}\E\bigg[\Big|\int_\Xd( U_{s,t}\varphi)(\mu_s^N-\mu_s)\Big|^m\bigg]~:\,~0\leq s \leq t,~\varphi\in C^\infty_c(\Xd),~\|\langle z\rangle^{-p}\varphi\|_{W^{2,q'}(\Xd)}\le1\bigg\},
\end{equation*}
for some exponent $\gamma>0$ to be chosen appropriately later, where we recall that $\{U_{s,t}\}_{0\le s\le t}$ is the backward linearized mean-field evolution defined in~\eqref{eq:def-Ust}. Given a test function $\varphi\in C^\infty_c(\Xd)$ with $\|\langle z\rangle^{-p}\varphi\|_{W^{2,q'}(\Xd)}\le1$, we find by It\^o's formula, for all $m \ge 2$ and $0 \leq s \leq t$,
\begin{multline}\label{eq:mf_dvpt_ito}
e^{\gamma m(t-s)}\E\bigg[\Big|\int_\Xd (U_{s,t}\varphi)(\mu_s^N-\mu_s)\Big|^m\bigg]
\le\, e^{\gamma m t}\E\bigg[\Big|\int_\Xd (U_{0,t}\varphi)(\mu_0^N-\mu_\circ)\Big|^m\bigg] \\
- \gamma m \int_0^s e^{\gamma m(t-\tau)}\E\bigg[\Big|\int_\Xd (U_{\tau,t}\varphi)(\mu_\tau^N-\mu_\tau)\Big|^m\bigg] \ddr \tau \\
+ \kappa m \int_0^s e^{\gamma m(t-\tau)}\E\bigg[\Big|\int_\Xd (U_{\tau,t}\varphi)(\mu_\tau^N-\mu_\tau)\Big|^{m-1}\Big|\int_{\Xd}\mu_\tau^N \nabla_{v}(U_{\tau,t}\varphi)\cdot \nabla W\ast(\mu_\tau^N-\mu_\tau)\Big|\bigg] \ddr \tau \\
+ \frac{m(m-1)}{N}\int_0^s e^{\gamma m(t-\tau)}\E\bigg[\Big|\int_\Xd (U_{\tau,t}\varphi)(\mu_\tau^N-\mu_\tau)\Big|^{m-2}\int_\Xd |\nabla_{v}(U_{\tau,t}\varphi)|^2\mu^N_\tau\bigg]\ddr \tau.
\end{multline}
We start by examining the penultimate right-hand side term.
Similarly as e.g.\@ in~\cite{Cornalba-Fischer-Raithel-2023}, it is convenient to decompose $\nabla W\ast(\mu_s^N-\mu_s)$ as a superposition of Fourier modes,
\begin{equation*}
\int_{\Xd}\mu_\tau^N \nabla_{v}(U_{\tau,t}\varphi)\cdot \nabla W\ast(\mu_\tau^N-\mu_\tau)
\,=\, C\int_{\mathbb{R}^d} i\xi \widehat{W}(\xi)\cdot\Big(\int_{\Xd}\nabla_{v}(U_{\tau,t}\varphi) e_{\xi}\mu_\tau^N\Big)\Big(\int_{\Xd}e_{-\xi}(\mu_\tau^N - \mu_\tau)\Big)\ddr \xi,
\end{equation*}
where we have set $e_\xi(x) := e^{ix\cdot\xi}$.
By the ergodic estimates of Lemma~\ref{lem:ergodic}, together with the Sobolev embedding with {$\ell>2d/q'$}, we find
\begin{equation}\label{eq:L-infty-bound-grad-phi}
\|\langle z\rangle^{-p}\nabla_v(U_{\tau,t}\varphi)e_\xi\|_{\Ld^\infty(\Xd)}\,\lesssim\,\|\langle z\rangle^{-p}\nabla_v(U_{\tau,t}\varphi)\|_{W^{\ell,q'}(\Xd)}
\,\lesssim\,e^{-\lambda_0p(t-\tau)}\|\langle z\rangle^{-p}\varphi\|_{W^{\ell+1,q'}(\Xd)}.
\end{equation}
Choosing $q'\ge pq'\gg1$ large enough, we can choose {$\ell>2d/q'$} with $\ell+1\le2$.
Hence, recalling the choice of $\varphi$ with $\|\langle z\rangle^{-p}\varphi\|_{W^{2,q'}(\Xd)}\le1$, we get
\[\Big|\int_{\Xd}\mu_\tau^N \nabla_{v}(U_{\tau,t}\varphi)\cdot \nabla W\ast(\mu_\tau^N-\mu_\tau)\Big|\,\lesssim\,e^{-\lambda_0p(t-\tau)}Q_p(\mu_\tau^N)\int_{\R^d}|\xi\widehat{W}(\xi)|\Big|\int_\Xd e_{-\xi}(\mu_\tau^N-\mu_\tau)\Big|\,\ddr\xi,\]
where the multiplicative constant only depends on $d,\beta,p,q,a$, $Q_2(\mu_\circ)$, and $\|W\|_{W^{d+3,\infty}(\R^d)}$.
Given a parameter $L\ge1$ to be chosen appropriately later, decomposing $Q_p(\mu_\tau^N)\le L+(Q_p(\mu_\tau^N)-L)_+$, noting that $|\int_\Xd e_{-\xi}(\mu_\tau^N-\mu_\tau)|\le2$ and that the assumptions on $W$ ensure $\int_{\R^d}|\xi\widehat{W}(\xi)|\ddr\xi\lesssim1$, we can bound
\begin{multline*}
\Big|\int_{\Xd}\mu_\tau^N \nabla_{v}(U_{\tau,t}\varphi)\cdot \nabla W\ast(\mu_\tau^N-\mu_\tau)\Big|
\,\lesssim\,Le^{-\lambda_0p(t-\tau)}\int_{\R^d}|\xi\widehat{W}(\xi)|\Big|\int_\Xd e_{-\xi}(\mu_\tau^N-\mu_\tau)\Big|\,\ddr\xi\\
+e^{-\lambda_0p(t-\tau)}(Q_p(\mu_\tau^N)-L)_+.
\end{multline*}
Now using this to estimate the penultimate term in~\eqref{eq:mf_dvpt_ito}, we find
\begin{multline}\label{eq:mf_dvpt_ito-re}
\int_0^s e^{\gamma m(t-\tau)}\E\bigg[\Big|\int_\Xd (U_{\tau,t}\varphi)(\mu_\tau^N-\mu_\tau)\Big|^{m-1}\Big|\int_{\Xd}\mu_\tau^N \nabla_{v}(U_{\tau,t}\varphi)\cdot \nabla W\ast(\mu_\tau^N-\mu_\tau)\Big|\bigg] \ddr \tau\\
\,\lesssim\,
L\int_{\R^d}|\xi\widehat W(\xi)| \int_0^s e^{\gamma m(t-\tau)}e^{-\lambda_0p(t-\tau)} \E\bigg[\Big|\int_\Xd e_{-\xi}(\mu_\tau^N-\mu_\tau)\Big|\Big|\int_\Xd (U_{\tau,t}\varphi)(\mu_\tau^N-\mu_\tau)\Big|^{m-1}\bigg]\ddr \tau\ddr\xi\\
+\int_0^s e^{\gamma m(t-\tau)}e^{-\lambda_0p(t-\tau)} \E\bigg[(Q_p(\mu_\tau^N)-L)_+\Big|\int_\Xd (U_{\tau,t}\varphi)(\mu_\tau^N-\mu_\tau)\Big|^{m-1}\bigg]\ddr \tau.
\end{multline}
Let us start by examining the first right-hand side term in this estimate.
Noting that the definition of~$D_{p,q}^{N,m}(t)$ entails in particular
\[\E\bigg[\Big|\int_\Xd e_{-\xi}(\mu_\tau^N-\mu_\tau)\Big|^m\bigg]^\frac1m\,\le\, \|\langle z\rangle^{-p}e_{-\xi}\|_{W^{2,q'}(\Xd)}D_{p,q}^{N,m}(\tau)^\frac1m\,\lesssim\,\langle\xi\rangle^2D_{p,q}^{N,m}(\tau)^\frac1m,\]
and noting that the assumptions on $W$ allow to estimate as follows the remaining integral with respect to $\xi$, for some $\delta>0$,
\begin{equation*}
\int_{\mathbb{R}^d} \langle \xi\rangle^2|\xi\widehat W(\xi)|\ddr \xi
\,\lesssim_\delta\,\Big(\int_{\mathbb{R}^d} \langle \xi\rangle^{2(\frac{d+\delta}{2}+3)} |\widehat{W}(\xi)|^2\ddr \xi\Big)^\frac12\,
\simeq\,\|W\|_{H^{\frac{d+\delta}{2}+3}(\R^d)}\,<\,\infty,
\end{equation*}
we can use H\"older's inequality to estimate
\begin{multline*}
\lefteqn{L\int_{\R^d}|\xi\widehat W(\xi)| \int_0^s e^{\gamma m(t-\tau)}e^{-\lambda_0p(t-\tau)} \E\bigg[\Big|\int_\Xd e_{-\xi}(\mu_\tau^N-\mu_\tau)\Big|\Big|\int_\Xd (U_{\tau,t}\varphi)(\mu_\tau^N-\mu_\tau)\Big|^{m-1}\bigg]\ddr \tau\ddr\xi}\\
\,\lesssim\,L\int_0^s e^{\gamma m(t-\tau)}e^{-\lambda_0p(t-\tau)} \E\bigg[\Big|\int_\Xd (U_{\tau,t}\varphi)(\mu_\tau^N-\mu_\tau)\Big|^m\bigg]^{1-\frac{1}m}D_{p,q}^{N,m}(\tau)^\frac1m\,\ddr \tau.
\end{multline*}
Inserting this into~\eqref{eq:mf_dvpt_ito-re} and further using Young's inequality, we get
\begin{multline}\label{eq:mf_dvpt_ito-1}
\int_0^s e^{\gamma m(t-\tau)}\E\bigg[\Big|\int_\Xd (U_{\tau,t}\varphi)(\mu_\tau^N-\mu_\tau)\Big|^{m-1}\Big|\int_{\Xd}\mu_\tau^N \nabla_{v}(U_{\tau,t}\varphi)\cdot \nabla W\ast(\mu_\tau^N-\mu_\tau)\Big|\bigg] \ddr \tau\\
\,\lesssim\,
L\int_0^s e^{\gamma m(t-\tau)}\E\bigg[\Big|\int_\Xd (U_{\tau,t}\varphi)(\mu_\tau^N-\mu_\tau)\Big|^m\bigg]\,\ddr \tau\\
+\frac{L}m\int_0^s e^{(\gamma-\lambda_0p)m(t-\tau)} D_{p,q}^{N,m}(\tau)\,\ddr \tau
+\frac1m\int_0^s e^{(\gamma -\lambda_0p)m(t-\tau)} \E\big[(Q_p(\mu_\tau^N)-L)_+^m\big]\ddr \tau.
\end{multline}
We turn to the last term in~\eqref{eq:mf_dvpt_ito}. Applying again Young's inequality, we can bound
\begin{multline*}
\frac{m(m-1)}{N}\int_0^se^{\gamma m(t-\tau)}\E\bigg[\Big|\int_\Xd (U_{\tau,t}\varphi)(\mu_\tau^N-\mu_\tau)\Big|^{m-2}\int_\Xd |\nabla_{v}(U_{\tau,t}\varphi)|^2\mu^N_\tau\bigg]\ddr \tau \\
\,\le\, \frac{\gamma m}{2}\int_0^se^{\gamma m(t-\tau)}\E\bigg[\Big|\int_\Xd U_{\tau,t}\varphi(\mu_\tau^N-\mu_\tau)\Big|^{m}\bigg]\ddr \tau
+ \gamma\Big(\frac{2m}{\gamma N}\Big)^\frac{m}{2}\int_0^s e^{\gamma m(t-\tau)} \E\bigg[\Big(\int_\Xd |\nabla_{v}(U_{\tau,t}\varphi)|^2\mu^N_\tau\Big)^\frac{m}{2}\bigg]\ddr \tau,
\end{multline*}
and thus, using again the ergodic estimates of Lemma~\ref{lem:ergodic} as in~\eqref{eq:L-infty-bound-grad-phi},
\begin{multline*}
\frac{m(m-1)}{N}\int_0^se^{\gamma m(t-\tau)}\E\bigg[\Big|\int_\Xd (U_{\tau,t}\varphi)(\mu_\tau^N-\mu_\tau)\Big|^{m-2}\int_\Xd |\nabla_{v}(U_{\tau,t}\varphi)|^2\mu^N_\tau\bigg]\ddr \tau \\
\,\le\, \frac{\gamma m}{2}\int_0^se^{\gamma m(t-\tau)}\E\bigg[\Big|\int_\Xd U_{\tau,t}\varphi(\mu_\tau^N-\mu_\tau)\Big|^{m}\bigg]\ddr \tau
+ \gamma\Big(\frac{Cm}{\gamma N}\Big)^\frac{m}{2}\int_0^s e^{(\gamma -\lambda_0p)m(t-\tau)}\E\big[Q_{2p}(\mu^N_\tau)^\frac{m}{2}\big]\ddr \tau.
\end{multline*}
Inserting this together with~\eqref{eq:mf_dvpt_ito-1} into~\eqref{eq:mf_dvpt_ito}, we are led to
\begin{multline*}
e^{\gamma m(t-s)}\E\bigg[\Big|\int_\Xd (U_{s,t}\varphi)(\mu_s^N-\mu_s)\Big|^m\bigg]
\le\, e^{\gamma m t}\E\bigg[\Big|\int_\Xd (U_{0,t}\varphi)(\mu_0^N-\mu_\circ)\Big|^m\bigg] \\
+\big(C\kappa L-\tfrac12\gamma\big) m \int_0^s e^{\gamma m(t-\tau)}\E\bigg[\Big|\int_\Xd (U_{\tau,t}\varphi)(\mu_\tau^N-\mu_\tau)\Big|^m\bigg] \ddr \tau \\
+C\kappa L\int_0^s e^{(\gamma-\lambda_0p)m(t-\tau)} D_{p,q}^{N,m}(\tau)\,\ddr \tau
+C\kappa\int_0^s e^{(\gamma -\lambda_0p)m(t-\tau)} \E\big[(Q_p(\mu_\tau^N)-L)_+^m\big]\ddr \tau\\
+ \gamma\Big(\frac{Cm}{\gamma N}\Big)^\frac{m}{2}\int_0^s e^{(\gamma -\lambda_0p)m(t-\tau)}\E\big[Q_{2p}(\mu^N_\tau)^\frac{m}{2}\big]\ddr \tau.
\end{multline*}
Note that for the first right-hand side term the ergodic estimates of Lemma~\ref{lem:ergodic} and the definition of~$D_{p,q}^{N,m}(0)$ yield
\begin{equation*}
\E\bigg[\Big|\int_\Xd (U_{0,t}\varphi)(\mu_0^N-\mu_\circ)\Big|^m\bigg] \,\le\, \Big\|\langle z\rangle^{-p}\Big(U_{0,t}\varphi - \int_{\Xd}\varphi M\Big)\Big\|^m_{W^{2,q'}(\Xd)}D_{p,q}^{N,m}(0)\,\le\, C^m e^{-\lambda_0 pm t }D_{p,q}^{N,m}(0).
\end{equation*}
Provided that $\gamma \geq 2C\kappa L$, taking the supremum over $\varphi$ and $0 \leq s \leq t$, we deduce
\begin{multline}\label{eq:estim-gronDNm}
D_{p,q}^{N,m}(t) \,\le\, C^me^{(\gamma -\lambda_0p)m t}D_{p,q}^{N,m}(0)
+C\kappa L \int_0^t e^{(\gamma - \lambda_0p)m(t-\tau)}D_{p,q}^{N,m}(\tau)\,\ddr \tau \\
+ C\kappa\int_0^t e^{(\gamma-\lambda_0p)m(t-\tau)}\E\big[(Q_{2p}(\mu^N_\tau)-L)^m_+\big] \ddr \tau \\
+ \gamma\Big(\frac{Cm}{\gamma N}\Big)^\frac{m}{2}\int_0^t e^{(\gamma -\lambda_0p)m(t-\tau)}\E\big[Q_{2p}(\mu^N_\tau)^\frac{m}{2}\big] \ddr \tau.
\end{multline}
Noting that for $pq'\gg1$ we have
\begin{equation*}
\E\bigg[\Big|\int_{\Xd} \langle z\rangle^{2p}(\mu_0^N - \mu_\circ)\Big|^m\bigg] \,\le\,C^mD_{3p,q}^{N,m}(0),
\end{equation*}
and choosing $L\gg1$ large enough (only depending on $d,\beta,a$, $Q_2(\mu_\circ)$, $\|W\|_{W^{1,\infty}(\R^d)}$),
the concentration bound of Lemma~\ref{lem:mom} yields, for $2\le m\le N$,
\begin{multline*}
\int_0^t e^{(\gamma-\lambda_0p)m(t-\tau)}\,\E\big[(Q_{2p}(\mu^N_\tau)-L)^m_+\big] \ddr \tau \\
\,\le\,
\Big(\frac{CLm}N\Big)^\frac{m}{2}\int_0^t e^{(\gamma-\lambda_0p)m(t-\tau)}\ddr \tau
+C^mD_{3p,q}^{N,m}(0)\int_0^t e^{(\gamma-\lambda_0p)m(t-\tau)}e^{-\lambda_0pm \tau}\ddr \tau,
\end{multline*}
and thus, provided that $\gamma\le\frac12\lambda_0p$,
\begin{equation*}
\int_0^t e^{(\gamma-\lambda_0p)m(t-\tau)}\,\E\big[(Q_{2p}(\mu^N_\tau)-L)^m_+\big] \ddr \tau 
\,\le\,
\Big(\frac{CLm}N\Big)^\frac{m}{2}
+C^me^{(\gamma-\lambda_0p)mt}D_{3p,q}^{N,m}(0).
\end{equation*}
Similarly, we can also bound
\begin{eqnarray*}
\lefteqn{\Big(\frac{Cm}N\Big)^\frac m2\int_0^t e^{(\gamma -\lambda_0p)m(t-\tau)}\E\big[Q_{2p}(\mu^N_\tau)^\frac{m}{2}\big] \ddr \tau}\\
&\le&\Big(\frac{CLm}N\Big)^\frac m2\int_0^t e^{(\gamma -\lambda_0p)m(t-\tau)}\ddr \tau
+C^m\int_0^t e^{(\gamma -\lambda_0p)m(t-\tau)}\E\big[(Q_{2p}(\mu^N_\tau)-L)_+^m\big] \ddr \tau\\
&\le&\Big(\frac{CLm}N\Big)^\frac m2
+C^me^{(\gamma-\lambda_0p)mt}D_{3p,q}^{N,m}(0).
\end{eqnarray*}
Inserting these bounds into~\eqref{eq:estim-gronDNm}, choosing $\gamma:=\frac12\lambda_0p$, and assuming that $\kappa$ is small enough in the sense of $2C\kappa L\le\frac12\lambda_0p$, we find
\begin{equation*}
D_{p,q}^{N,m}(t) \,\le\,\Big(\frac{CLm}N\Big)^\frac m2
+C^me^{-\frac12\lambda_0pm t}D_{3p,q}^{N,m}(0)
+C\kappa L \int_0^t e^{-\frac12\lambda_0pm(t-\tau)}D_{p,q}^{N,m}(\tau)\,\ddr \tau,
\end{equation*}
and thus, by Gr\"onwall's inequality,
\begin{multline*}
D_{p,q}^{N,m}(t) \,\le\,
\Big(\frac{CLm}N\Big)^\frac m2
+C^me^{-\frac12\lambda_0pm t}D_{3p,q}^{N,m}(0)\\
+C\kappa L e^{-\frac12\lambda_0pm t}\int_0^t e^{C\kappa L(t-\tau)}\bigg(e^{\frac12\lambda_0pm \tau}\Big(\frac{CLm}N\Big)^\frac m2
+C^mD_{3p,q}^{N,m}(0)\bigg)\,\ddr \tau.
\end{multline*}
As the smallness requirement on $\kappa$ precisely ensures $C\kappa L \le \frac{1}{4}\lambda_0 p$, we get after computing the remaining time integrals,
\begin{equation*}
D_{p,q}^{N,m}(t) \,\le\,
\Big(\frac{CLm}N\Big)^\frac m2
+C^m e^{-\frac14\lambda_0p m t}D_{3p,q}^{N,m}(0).
\end{equation*}
As by definition we have $A^{N,m}_{p,q}(t) \le D^{N,m}_{p,q}(t)$ and $A^{N,m}_{p,q}(0) = D^{N,m}_{p,q}(0)$, the conclusion follows (up to renaming $\lambda_0$).
\end{proof}

\section{Creation of chaos}
\label{sec:proof_thm_main}
This section is devoted to the proof of Theorem~\ref{th:creation} using the preliminary results above. Our argument shows how the estimates of Proposition~\ref{prop:ito} can be turned into suboptimal mean-field estimates for marginals, which in turn can be made optimal in $N$ after combination with direct estimates on the BBGKY hierarchy. A similar idea will be used again for the proof of Theorem~\ref{th:estim-GNm} in the next section.

\begin{proof}[Proof of Theorem~\ref{th:creation}]
We split the proof into three steps.
Let $\lambda_0,p,q,\kappa$ be as in Proposition~\ref{prop:ito}, and assume $p\le\frac16$. In the sequel of this proof, all multiplicative constants are allowed to depend on $d,\beta,a,p,q,$ $Q_2(\mu_\circ)$, $Q_{2}(F_\circ^{N,1})$, and $\|W\|_{W^{d+3,\infty}\cap H^s(\R^d)}$ for some $s>\frac d2+5$. A subscript `$m$' to constants is used to indicate further dependence on $m$ and on $Q_{mp}(F_\circ^{N,1})$.

\medskip\noindent
{\bf Step~1:} Reformulation of Proposition~\ref{prop:ito}: proof of the suboptimal mean-field estimate
\begin{multline}
\label{eq:step1suboptimal}
\|F_t^{N,1}-\mu_t\|_{W^{-2,q}_p(\Xd)}^2+\|G_t^{N,2}\|_{W^{-2,q}_p(\Xd)^{\otimes 2}}\\
\,\lesssim\,N^{-1}+e^{-2\lambda_0pt}\Big(\|F_\circ^{N,1}-\mu_\circ\|_{W^{-2,q}_{3p}(\Xd)}^2+\|G_\circ^{N,2}\|_{W^{-2,q}_{3p}(\Xd)^{\otimes 2}}\Big).
\end{multline}
Noting that
\begin{equation*}
\E\bigg[\Big(\int_{\Xd} \varphi\,(\mu^N_t-\mu_t)\Big)\bigg]
\,=\,\int_\Xd\varphi (F^{N,1}_t-\mu_t),
\end{equation*}
and that
\begin{eqnarray*}
A^{N,2}_{p,q}(t)
&\ge&\sup\bigg\{\Big(\int_\Xd\varphi(F_t^{N,1}-\mu_t)\Big)^2~:~\varphi\in C^\infty_c(\Xd),~\|\langle \cdot \rangle^{-p}\varphi\|_{W^{2,q'}(\Xd)}\le 1\bigg\}\\
&=&\|F_t^{N,1}-\mu_t\|_{W^{-2,q}_p(\Xd)}^2,
\end{eqnarray*}
Proposition~\ref{prop:ito} implies for $m=2$,
\begin{equation}\label{eq:reform-B1}
\|F_t^{N,1}-\mu_t\|_{W^{-2,q}_p(\Xd)}^2\,\lesssim\,N^{-1}+e^{-2\lambda_0pt}A_{3p,q}^{N,2}(0).
\end{equation}
Also note that
\begin{equation*}
\E\bigg[\Big(\int_{\Xd} \varphi\,(\mu^N_t-\mu_t)\Big)^2\bigg]
\,=\,\Big(\int_\Xd\varphi (F^{N,1}_t-\mu_t)\Big)^2
+\int_\Xd\varphi^{\otimes2}G^{N,2}_t
+\tfrac1N\int_\Xd\varphi^2F^{N,1}_t
-\tfrac1N\int_\Xd\varphi^{\otimes2}F^{N,2}_t,
\end{equation*}
and thus, taking the supremum over $\varphi$, and using the Sobolev embedding together with the moment bounds of Lemma~\ref{lem:mom}, which yields in particular
\begin{equation}\label{eq:control_QF2}
\|F_t^{N,2}\|_{W_p^{-2,q}(\Xd)^{\otimes2}}\,\lesssim\,\int_{\Xd^2}\langle z\rangle^p\langle z_*\rangle^pF_t^{N,2}(z,z_*)\,\ddr z\ddr z_*\,\le\,\int_{\Xd}\langle z\rangle^{2p}F_t^{N,1} = Q_{2p}(F^{N,1}_t),
\end{equation}
we get
\begin{equation*}
\Big|A_{p,q}^{N,2}(t)-\|G_t^{N,2}\|_{W^{-2,q}_{p}(\Xd)^{\otimes 2}}\Big|\,\le\,\|F_t^{N,1}-\mu_t\|^2_{W^{-2,q}_{p}(\Xd)}+CN^{-1}Q_{2p}(F_t^{N,1}).
\end{equation*}
Appealing to Lemma~\ref{lem:mom} in form of $Q_{2p}(F_t^{N,1})\lesssim 1 + Q_{2p}(F_\circ^{N,1})\lesssim1$,
this means
\begin{equation}\label{eq:AtN2-equiv}
\Big|A_{p,q}^{N,2}(t)-\|G_t^{N,2}\|_{W^{-2,q}_{p}(\Xd)^{\otimes 2}}\Big|\,\lesssim\,\|F_t^{N,1}-\mu_t\|^2_{W^{-2,q}_{p}(\Xd)}+CN^{-1}.
\end{equation}
Combining with the result of Proposition~\ref{prop:ito} to get a control on $G_t^{N,2}$, we obtain
\[\|G_t^{N,2}\|_{W^{-2,q}_p(\Xd)^{\otimes 2}}\,\lesssim\,\|F_t^{N,1}-\mu_t\|_{W^{-2,q}_p(\Xd)}^2+N^{-1}+e^{-2\lambda_0pt}A^{N,2}_{3p,q}(0),\]
which thus yields, together with~\eqref{eq:reform-B1},
\begin{equation}\label{eq:bound_G_N_2_A_2}
\|F_t^{N,1}-\mu_t\|_{W^{-2,q}_p(\Xd)}^2+\|G_t^{N,2}\|_{W^{-2,q}_p(\Xd)^{\otimes 2}}\,\lesssim\,N^{-1}+e^{-2\lambda_0pt}A^{N,2}_{3p,q}(0).
\end{equation}
Now applying~\eqref{eq:AtN2-equiv} at $t=0$ with $p$ replaced by $3p$, we find
\[A_{3p,q}^{N,2}(0)\,\le\,\|F_\circ^{N,1}-\mu_\circ\|^2_{W^{-2,q}_{3p}(\Xd)}+\|G_\circ^{N,2}\|_{W^{-2,q}_{3p}(\Xd)^{\otimes 2}}+CN^{-1},\]
and the claim~\eqref{eq:step1suboptimal} follows.

\medskip\noindent
{\bf Step~2:} Proof of the improved mean-field bound
\begin{eqnarray}
\lefteqn{\|F_t^{N,1}-\mu_t\|_{W^{-3,q}_p(\Xd)}}\nonumber\\
&\lesssim&N^{-1} + e^{-\frac12\lambda_0pt}\Big(\|F_\circ^{N,1} - \mu_\circ\|_{W^{-2,q}_{p}(\Xd)}+\|F_\circ^{N,1} - \mu_\circ\|_{W^{-2,q}_{3p}(\Xd)}^2+\|G^{N,2}_\circ\|_{W^{-2,q}_{3p}(\Xd)^{\otimes 2}}\Big)\nonumber\\
&\lesssim& N^{-1} + e^{-\frac12\lambda_0pt}\Big(\|F_\circ^{N,1} - \mu_\circ\|_{W^{-2,q}_{3p}(\Xd)}+\|G^{N,2}_\circ\|_{W^{-2,q}_{3p}(\Xd)^{\otimes 2}}\Big).\label{eq:step2optimal}
\end{eqnarray}
Integrating the Liouville equation~\eqref{eq:Liouville} with respect to its last $N-1$ coordinates, and decomposing the second marginal as $F^{N,2}=F^{N,1}\otimes F^{N,1}+G^{N,2}$, we find that the first marginal $F^{N,1}$ satisfies the following BBGKY equation,
\begin{multline*}
\partial_t(F^{N,1}-\mu)=R_\mu(F^{N,1}-\mu)+\kappa \nabla W\ast(F^{N,1}-\mu)\cdot\nabla_vF^{N,1}\\
+\kappa\int_\Xd \nabla W(\cdot-x_*)\cdot\nabla_v(G^{N,2}-\tfrac1NF^{N,2})(\cdot,z_*)\,\ddr z_*,
\end{multline*}
where $R_\mu$ stands for the following (truncated) linearized mean-field operator (linearized at the mean-field solution~$\mu$),
\[R_\mu h\,:=\,\tfrac12\Div_v((\nabla_v+\beta v)h)-v\cdot\nabla_xh+(\nabla A+\kappa\nabla W\ast\mu)\cdot\nabla_vh.\]
Given $\varphi \in C^\infty_c(\Xd)$, recalling the definition~\eqref{eq:def-Ust} of the backward linearized evolution $\{U_{s,t}\}_{0\le s\le t}$, we may then compute
\begin{multline}\label{eq:decomp-F1-mu-hier}
\int_{\Xd} \varphi(F^{N,1}_t - \mu_t) 
\,=\,\int_{\Xd} (U_{0,t}\varphi)(F^{N,1}_\circ - \mu_\circ)
-\kappa\int_0^t\Big( \int_{\Xd}\nabla_v (U_{s,t}\varphi)\cdot \nabla W \ast(F^{N,1}_s - \mu_s)\,F^{N,1}_s\Big)\ddr s\\
-\kappa \int_0^t\Big(\int_{\Xd^2}\nabla_v (U_{s,t}\varphi)(z)\cdot\nabla W(x-x_\ast)\,(G^{N,2}_s-\tfrac{1}{N}F^{N,2}_s)(z,z_\ast)\,\ddr z \ddr z_\ast\Big)\ddr s.
\end{multline}
Let us examine the three right-hand side terms separately.
Using the ergodic estimates of Lemma~\ref{lem:ergodic}, the first one is bounded by
\[\Big|\int_{\Xd} (U_{0,t}\varphi)(F^{N,1}_\circ - \mu_\circ)\Big|\,\lesssim\,e^{-\lambda_0pt}\|\langle z\rangle^{-p}\varphi\|_{W^{2,q'}(\Xd)}\|F^{N,1}_\circ - \mu_\circ\|_{W^{-2,q}_p(\Xd)}.\]
For the second term, using again the ergodic estimates of Lemma~\ref{lem:ergodic} along with the Sobolev embedding, we can bound
\begin{eqnarray*}
\lefteqn{\Big|\int_{\Xd}\nabla_v (U_{s,t}\varphi)\cdot \nabla W \ast(F^{N,1}_s - \mu_s)\,F^{N,1}_s\Big|}\\
&\lesssim&Q_p(F^{N,1}_s)\,\|\langle z\rangle^{-p}\nabla_v (U_{s,t}\varphi)\|_{\Ld^\infty(\Xd)}\sup_{x\in\R^d}\Big|\int_\Xd \nabla W(x-\cdot)\,(F^{N,1}_s - \mu_s)\Big|\\
&\lesssim&Q_p(F^{N,1}_s)\,e^{-\lambda_0p(t-s)}\,\|\langle z\rangle^{-p}\varphi\|_{W^{2,q'}(\Xd)}\,\|F^{N,1}_s - \mu_s\|_{W^{-3,q}_p(\Xd)},
\end{eqnarray*}
and thus, further using the moment bounds of Lemma~\ref{lem:mom},
\begin{equation*}
\Big|\int_{\Xd}\nabla_v (U_{s,t}\varphi)\cdot \nabla W \ast(F^{N,1}_s - \mu_s)\,F^{N,1}_s\Big|
\,\lesssim\,e^{-\lambda_0p(t-s)}\,\|\langle z\rangle^{-p}\varphi\|_{W^{2,q'}(\Xd)}\,\|F^{N,1}_s - \mu_s\|_{W^{-3,q}_p(\Xd)}.
\end{equation*}
For the last term in~\eqref{eq:decomp-F1-mu-hier}, we aim to appeal to the bound~\eqref{eq:step1suboptimal} of Step~1 for $G^{N,2}$. Note, however, that the map $(z,z_*)\mapsto\nabla_v(U_{s,t}\varphi)(z)\cdot\nabla W(x-x_*)$ is not an admissible test function for such estimates in~$W^{-2,q}_p(\Xd)^{\otimes2}$, cf.~\eqref{eq:definition_W_tilde}. To remedy this, we decompose $W$ as a superposition of Fourier modes, which indeed allows to reduce to a tensorized test function. More precisely, by polarization, using that $G^{N,2}_s-\tfrac{1}{N}F^{N,2}_s$ is symmetric in its two variables, and recalling the short-hand notation $e_\xi(x)=e^{ix\cdot\xi}$, we can bound
\begin{equation}\label{eq:bound_fourier_G_N_2}
\begin{aligned}
\lefteqn{\Big|\int_{\Xd^2}\nabla_v (U_{s,t}\varphi)(z)\cdot\nabla W(x-x_\ast)\,(G^{N,2}_s-\tfrac{1}{N}F^{N,2}_s)(z,z_\ast)\,\ddr z \ddr z_\ast\Big|}\\
&\le\int_{\R^d}|\xi\widehat W(\xi)|\Big|\int_{\Xd^2}e_\xi(x_*)e_{-\xi}(x) \nabla_v (U_{s,t}\varphi)(z) \,(G^{N,2}_s-\tfrac{1}{N}F^{N,2}_s)(z,z_\ast)\,\ddr z \ddr z_\ast\Big|\ddr\xi\\
&\le4\|G^{N,2}_s-\tfrac{1}{N}F^{N,2}_s\|_{W^{-2,q}_p(\Xd)^{\otimes2}}\int_{\R^d}|\xi\widehat W(\xi)|\|\langle z\rangle^{-p}e_\xi\|_{W^{2,q'}(\Xd)}\|\langle z\rangle^{-p}e_\xi\nabla_v(U_{s,t}\varphi)\|_{W^{2,q'}(\Xd)}\,\ddr\xi.
\end{aligned} 
\end{equation}
Noting that
\[\|\langle z\rangle^{-p}e_\xi\|_{W^{2,q'}(\Xd)}\,\lesssim\,\langle\xi\rangle^2,\qquad\|\langle z\rangle^{-p}e_\xi\nabla_v(U_{s,t}\varphi)\|_{W^{2,q'}(\Xd)}\,\lesssim\,\langle\xi\rangle^2e^{-\lambda_0p(t-s)}\|\langle z\rangle^{-p}\varphi\|_{W^{3,q'}(\Xd)},\]
we deduce
\begin{multline*}
\Big|\int_{\Xd^2}\nabla_v (U_{s,t}\varphi)(z)\cdot\nabla W(x-x_\ast)\,(G^{N,2}_s-\tfrac{1}{N}F^{N,2}_s)(z,z_\ast)\,\ddr z \ddr z_\ast\Big|\\
\,\lesssim\,e^{-\lambda_0p(t-s)}\|\langle z\rangle^{-p}\varphi\|_{W^{3,q'}(\Xd)}\|G^{N,2}_s-\tfrac{1}{N}F^{N,2}_s\|_{W^{-2,q}_p(\Xd)^{\otimes2}}\int_{\R^d}\langle\xi\rangle^4|\xi\widehat W(\xi)|\,\ddr\xi,
\end{multline*}
where the last factor can be controlled by $\|W\|_{H^s(\R^d)}$ for $s>\frac d2+5$.
Combining these estimates for the different terms in~\eqref{eq:decomp-F1-mu-hier}, and taking the supremum over $\varphi$,
we get
\begin{multline*}
\|F^{N,1}_t - \mu_t\|_{W^{-3,q}_p(\Xd)}
\,\lesssim\,e^{-\lambda_0pt}\|F_\circ^{N,1}-\mu_\circ\|_{W^{-2,q}_p(\Xd)}\\
+\kappa\int_0^te^{-\lambda_0p(t-s)}\|F_s^{N,1}-\mu_s\|_{W^{-3,q}_p(\Xd)}\,\ddr s+\kappa\int_0^te^{-\lambda_0p(t-s)}\|G_s^{N,2}-\tfrac1NF^{N,2}_s\|_{W^{-2,q}_p(\Xd)^{\otimes2}}\,\ddr s.
\end{multline*}
Now inserting the bound~\eqref{eq:step1suboptimal} of Step~1 for $G^{N,2}$, using~\eqref{eq:control_QF2} together with $Q_{2p}(F^{N,1}_t) \lesssim 1$, we obtain
\begin{multline*}
\|F^{N,1}_t - \mu_t\|_{W^{-3,q}_p(\Xd)}
\,\le\,CN^{-1}
+C\kappa\int_0^te^{-\lambda_0p(t-s)}\|F_s^{N,1}-\mu_s\|_{W^{-3,q}_p(\Xd)}\,\ddr s\\
+Ce^{-\lambda_0pt}\Big(\|F_\circ^{N,1}-\mu_\circ\|_{W^{-2,q}_{p}(\Xd)}+\|F_\circ^{N,1}-\mu_\circ\|_{W^{-2,q}_{3p}(\Xd)}^2+\|G_\circ^{N,2}\|_{W^{-2,q}_{3p}(\Xd)^{\otimes2}}\Big).
\end{multline*}
Provided that $\kappa$ is chosen small enough so that $C\kappa<\frac12\lambda_0p$, the claim~\eqref{eq:step2optimal} follows from Gr\"onwall's inequality.

\medskip\noindent
{\bf Step~3:} Conclusion.\\
Let $\varphi\in C^\infty_c(\Xd)$ be momentarily fixed with $
\|\langle z\rangle^{-p}\varphi\|_{W^{3,q'}(\Xd)}\le1$.
Moments of the empirical measure can be computed as follows, for all $m\ge1$,
\begin{eqnarray*}
\expecM{\Big(\int_\Xd\varphi(\mu^N_t-\mu_t)\Big)^m}
&=&N^{-m}\sum_{i_1,\ldots,i_m=1}^N\expecM{\prod_{l=1}^m\Big(\varphi(Z_t^{i_l,N})-{\int_\Xd}\varphi\mu_t\Big)}\\
&=&N^{-m}\sum_{\pi\vdash\llbracket m\rrbracket}N(N-1)\ldots(N-\sharp\pi+1)\int_{\Xd^{\sharp\pi}}\bigg(\bigotimes_{B\in\pi}\Big(\varphi-{\int_\Xd}\varphi\mu_t\Big)^{\sharp B}\bigg)F_t^{N,\sharp\pi},
\end{eqnarray*}
where we use the same notation as in~\eqref{eq:cluster-exp0} for sums over partitions.
Note that the moment bounds of Lemma~\ref{lem:mom} together with the Sobolev embedding and the choice of~$\varphi$ yield, for all $\pi\vdash\llbracket m\rrbracket$, mimicking~\eqref{eq:control_QF2},
\[\bigg|\int_{\Xd^{\sharp\pi}}\bigg(\bigotimes_{B\in\pi}\Big(\varphi-{\int_\Xd}\varphi\mu_t\Big)^{\sharp B}\bigg)F_t^{N,\sharp\pi}\bigg|\,\lesssim_m\,\int_{\Xd}\langle z\rangle^{mp}F^{N,1}_t\,\le\,(Cm)^\frac{mp}2+C^me^{-\lambda_0pmt}Q_{mp}(F^{N,1}_\circ)\,\lesssim_m\,1.\]
Hence, neglecting all $O(N^{-1})$ terms in the above, we get
\begin{equation*}
\bigg|\E\bigg[\Big(\int_{\Xd} \varphi(\mu^N_t-\mu_t)\Big)^{m}\bigg]
\,-\,\int_{\Xd^m} \Big(\varphi-\int_\Xd\varphi\mu_t\Big)^{\otimes m}F_t^{N,m}\bigg|\,\lesssim_m\, N^{-1}.
\end{equation*}
Noting that
\begin{eqnarray*}
\int_{\Xd^m} \Big(\varphi-\int_\Xd\varphi\mu_t\Big)^{\otimes m}F_t^{N,m}&=&\sum_{j=0}^m\binom{m}j(-1)^{m-j}\Big(\int_\Xd\varphi\mu_t\Big)^{m-j}\int_{\Xd^j} \varphi^{\otimes j}F_t^{N,j}\\
&=&\sum_{j=1}^m\binom{m}j(-1)^{m-j}\Big(\int_\Xd\varphi\mu_t\Big)^{m-j}\int_{\Xd^j} \varphi^{\otimes j}(F_t^{N,j}-\mu_t^{\otimes j}),
\end{eqnarray*}
and focusing on the term corresponding to $j=m$ in the sum, we deduce
\begin{equation}\label{eq:estim-mutN-mut}
\bigg|\E\bigg[\Big(\int_{\Xd} \varphi(\mu^N_t-\mu_t)\Big)^{m}\bigg]
\,-\,\int_{\Xd^m} \varphi^{\otimes m}(F_t^{N,m}-\mu_t^{\otimes m})\bigg|\,\lesssim_m\, N^{-1}+\sum_{j=1}^{m-1}\Big|\int_{\Xd^j} \varphi^{\otimes j}(F_t^{N,j}-\mu_t^{\otimes j})\Big|.
\end{equation}
Taking the supremum over $\varphi \in C^{\infty}_c(\Xd)$ with $\|\langle z\rangle^{-p}\varphi\|_{W^{3,q'}(\Xd)} \le 1$, we deduce in particular for all~\mbox{$m\ge1$,}
\begin{equation}\label{eq:estim-Ftn-AtN}
\|F_t^{N,m}-\mu_t^{\otimes m}\|_{W_p^{-3,q}(\Xd)^{\otimes m}}
\,\lesssim_m\, N^{-1}+A_{p,q}^{N,m}(t)+\sum_{j=1}^{m-1}\|F_t^{N,j}-\mu_t^{\otimes j}\|_{W_p^{-3,q}(\Xd)^{\otimes j}},
\end{equation}
and thus, after a direct iteration,
\begin{equation*}
\|F_t^{N,m}-\mu_t^{\otimes m}\|_{W_p^{-3,q}(\Xd)^{\otimes m}}
\,\lesssim_m\, N^{-1}+\sum_{j=2}^m A_{p,q}^{N,j}(t)+\|F_t^{N,1}-\mu_t\|_{W^{-3,q}_p(\Xd)}.
\end{equation*}
Inserting the result of Proposition~\ref{prop:ito} and the result~\eqref{eq:step2optimal} of Step~2, we then get
\begin{multline}\label{eq:estim-pre-FNm-mum}
\|F_t^{N,m}-\mu_t^{\otimes m}\|_{W_p^{-3,q}(\Xd)^{\otimes m}}
\,\lesssim_m\, N^{-1}
+\sum_{j=2}^m e^{-\lambda_0pjt}A_{3p,q}^{N,j}(0)\\
+e^{-\frac12\lambda_0pt}\Big(\|F_\circ^{N,1}-\mu_\circ\|_{W^{-2,q}_{3p}(\Xd)}+\|G^{N,2}_\circ\|_{W^{-2,q}_{3p}(\Xd)^{\otimes2}}\Big).
\end{multline}
From here, it remains to express the $A_{3p,q}^{N,j}(0)$'s in terms of initial correlations. Noting that
\begin{equation}\label{eq:estim-ANm-mom0}
\E\bigg[\Big|\int_\Xd\varphi(\mu^N-\mu)\Big|^j\bigg]\,\le\,\left\{\begin{array}{lll}
\E\Big[\big(\int_\Xd\varphi(\mu^N-\mu)\big)^j\Big]&:&j\,\text{even},\\[3mm]
\E\Big[\big(\int_\Xd\varphi(\mu^N-\mu)\big)^{j-1}\Big]^\frac12\E\Big[\big(\int_\Xd\varphi(\mu^N-\mu)\big)^{j+1}\Big]^\frac12&:&j\,\text{odd},
\end{array}\right.
\end{equation}
using~\eqref{eq:estim-mutN-mut} at $t=0$, and taking the supremum over $\varphi \in C^{\infty}_c(\Xd)$ with $\|\langle z \rangle^{-3p} \varphi\|_{W^{2,q'}(\Xd)} \le 1$, we find for all $j\ge2$,
\begin{equation*}
A^{N,j}_{3p,q}(0)\,\lesssim_j\,
N^{-1}+\sum_{l=1}^{\tilde j}\|F_\circ^{N,l}-\mu_\circ^{\otimes l}\|_{W^{-2,q}_{3p}(\Xd)^{\otimes l}},\qquad \tilde j\,:=\,j+\mathds1_{j\,\text{odd}}.
\end{equation*}
Inserting this into~\eqref{eq:estim-pre-FNm-mum}, we get
\begin{equation*}
\|F_t^{N,m}-\mu_t^{\otimes m}\|_{W_p^{-3,q}(\Xd)^{\otimes m}}
\,\lesssim_m\, N^{-1}
+e^{-\frac12\lambda_0pt}\|G^{N,2}_\circ\|_{W^{-2,q}_{3p}(\Xd)^{\otimes2}}
+\sum_{j=1}^{\tilde m} e^{-\frac12\lambda_0pjt}\|F_\circ^{N,j}-\mu_\circ^{\otimes j}\|_{W^{-2,q}_{3p}(\Xd)^{\otimes j}},
\end{equation*}
with $\tilde m=m+\mathds1_{m\,\text{odd}}$.
Decomposing
\begin{equation}\label{eq:decomp-G2}
G_\circ^{N,2}=(F_\circ^{N,2}-\mu_\circ^{\otimes2})-(F_\circ^{N,1}-\mu_\circ)\otimes F_\circ^{N,1}-\mu_\circ\otimes (F_\circ^{N,1}-\mu_\circ),
\end{equation}
the conclusion~\eqref{eq:res-creation} follows for all $m\ge1$ (up to renaming $\lambda_0$).
\end{proof}

\section{Correlation estimates}\label{sec:correl}

This section is devoted to the proof of Theorem~\ref{th:estim-GNm}.
Similarly as for Theorem~\ref{th:creation}, our argument starts from the estimates of Proposition~\ref{prop:ito}: we show how these can be used to deduce suboptimal estimates on {correlations}, which in turn can be made optimal in $N$ after combination with direct estimates on the BBGKY hierarchy. Note that this time we will need to iterate BBGKY estimates multiple times to achieve optimality.
We split the proof into the next three subsections.

While in the introduction we defined correlation functions so as to satisfy the cluster expansion~\eqref{eq:cluster-exp0}, we recall that they can be defined more explicitly as polynomial combinations of marginals:
\begin{eqnarray*}
G^{N,1}&=&F^{N,1},\\
G^{N,2}&=&F^{N,2}-F^{N,1}\otimes F^{N,1},\\
G^{N,3}&=&\Sym\big(F^{N,3}-3F^{N,2}\otimes F^{N,1}+2(F^{N,1})^{\otimes 3}\big),\\
G^{N,4}&=&\Sym\big(F^{N,4}-4F^{N,3}\otimes F^{N,1}-3F^{N,2}\otimes F^{N,2}+12F^{N,2}\otimes (F^{N,1})^{\otimes 2}-6(F^{N,1})^{\otimes 4} \big),
\end{eqnarray*}
and so on, where the symbol `$\Sym$' stands for the symmetrization of coordinates. More generally, we can write for all $2\le m\le N$,
\begin{equation}\label{eq:def-cumGm}
G^{N,m}(z_1,\ldots,z_m)\,=\,\sum_{\pi\vdash \llbracket m \rrbracket}(\sharp\pi-1)!(-1)^{\sharp\pi-1}\prod_{B\in\pi}F^{N,\sharp B}(z_B),
\end{equation}
where we use a similar notation as in~\eqref{eq:cluster-exp0} for sums over partitions and where $\sharp\pi$ stands for the number of blocks in a partition~$\pi$.
We also use the following abbreviation for functional spaces: for all $r,p, q,m$,
\[W_{p,m}^{-r,q} \,:=\, W_p^{-r,q}(\Xd)^{\otimes m}.\]

\subsection{From moment mean-field estimates to correlations}
From now on, we let $\mu$ be the weak solution of the mean-field equation~\eqref{eq:VFP} with initial condition
\[\mu_\circ\,:=\,F^{N,1}_\circ.\]
In this section, we examine how the correlation functions $\{G^{N,m}\}_{2\le m\le N}$ can be related to the moments of $\mu_t^N-\mu_t$ in terms of the $A_{p,q}^{N,m}$'s defined in~\eqref{eq:def-ANm}, and we show that Proposition~\ref{prop:ito} implies the following hierarchy of estimates on correlations with suboptimal $N$-dependence ($N^{-m/2}$ instead of~$N^{1-m}$, cf.~\eqref{eq:refined-prop-chaos}).

\begin{prop}
\label{prop:basic_control_Gm}
Let $\lambda_0,p,q,\kappa$ be as in Proposition~\ref{prop:ito} with $\mu_\circ=F^{N,1}_\circ$.
Then we have for all $2\le m\le N$, $r \ge 2$, $0<p_0\le p$, and~$t\ge0$,
\begin{multline*}
\|G^{N,m}_t\|_{W^{-r,q}_{p_0,m}}\,\lesssim_{m}\,N^{-\frac m2}
+\sum_{1\le k\le l\le m-1}\sum_{j_1,\ldots,j_k\ge1\atop j_1+\ldots+j_k=l}N^{l-k-m+1}\prod_{i=1}^k\|G_t^{N,j_i}\|_{W^{-r,q}_{mp_0,j_i}}\\
+e^{-\lambda_0pmt}\times\left\{\begin{array}{lll}
\sum_{\ell=1}^{m/2}\sum_{\ell\le K\le L\le m}\sum_{j_1, \ldots, j_{K} \ge 1 \atop j_1 + \ldots j_{K} = L} N^{L-K-m+\ell} \prod_{i=1}^{K} \|G^{N,j_i}_\circ\|_{W^{-2,q}_{3mp, j_i}}&:&\text{$m$ even},\\
\sum_{\ell=1}^m\sum_{\ell\le K\le L\le2m}\sum_{1\le j_1,\ldots,j_K\le m+1\atop j_1+\ldots+j_K=L}\Big(N^{L-K-2m+\ell}\prod_{i=1}^K\|G_\circ^{N,j_i}\|_{W^{-2,q}_{3mp,j_i}}\Big)^{\frac12}&:&\text{$m$ odd},\\
\end{array}\right.
\end{multline*}
where the multiplicative constant only depends on $d,\beta,a,p_0,p,q,r,m$, $Q_2(F^{N,1}_\circ)$, and $\|W\|_{W^{d+3,\infty}\cap H^s(\R^d)}$ for some $s>\frac d2+3$.
\end{prop}

\begin{proof}
As a starting point, we recall the standard link between cumulants of the empirical measure and correlation functions.
We first used this link in~\cite[Section~4]{MD-21} and we further included a self-contained statement and proof in~\cite[Lemma~2.6]{BD_2024}. It can be stated as follows: for all $\varphi\in C^\infty_c(\Xd)$, $1\le m\le N$, and~$t \ge 0$,
\begin{multline}\label{eq:link-cum-G}
\bigg|\kappa^m\Big[\int_\Xd\varphi\mu^N_t\Big]
-\int_{\Xd^{m}}\varphi^{\otimes m} G^{N,m}_t\bigg|\\
\,\lesssim_m\,\sum_{\pi\vdash\llbracket m\rrbracket \atop \sharp\pi<m}\sum_{\rho\vdash\pi} N^{\sharp\pi-\sharp\rho-m+1}\bigg|\int_{\Xd^{\sharp\pi}}\Big(\bigotimes_{B\in\pi}\varphi^{\sharp B}\Big)\Big(\bigotimes_{D\in\rho} G^{N,\sharp D}_t(z_D)\Big)\,\ddr z_\pi\bigg|,
\end{multline}
where $\kappa^m[\cdot]$ stands for the $m$-th cumulant.
For $m\ge2$, as cumulants are invariant under translation, we can write $\kappa^m[\int_\Xd\varphi\mu_t^N]=\kappa^m[\int_\Xd\varphi(\mu_t^N-\mu_t)]$. By the definition of cumulants in terms of moments, we also recall for any random variable $X$,
\begin{equation}\label{eq:cumfrommom-inv}
\kappa^m[X]\,=\,\sum_{\pi\vdash \llbracket m \rrbracket}(-1)^{\sharp\pi-1}(\sharp\pi-1)!\prod_{B\in\pi}\E[X^{\sharp B}]\,\lesssim_m\,\E[|X|^m].
\end{equation}
Hence, from the above, we deduce in particular for all $\varphi\in C^\infty_c(\Xd)$, $2\le m\le N$, and $t\ge0$,
\begin{multline*}
\Big|\int_{\Xd^{m}}\varphi^{\otimes m}G^{N,m}_t\Big|
\,\lesssim_m\,\E\bigg[\Big|\int_\Xd\varphi(\mu^N_t-\mu_t)\Big|^m\bigg]\\
+\sum_{\pi\vdash\llbracket m\rrbracket \atop \sharp\pi<m}\sum_{\rho\vdash\pi} N^{\sharp\pi-\sharp\rho-m+1}\bigg|\int_{\Xd^{\sharp\pi}}\Big(\bigotimes_{B\in\pi}\varphi^{\sharp B}\Big)\Big(\bigotimes_{D\in\rho} G^{N,\sharp D}_t(z_D)\Big)\,\ddr z_\pi\bigg|.
\end{multline*}
For $r\ge2$, $1<q\le 2$ with $q'>d$, and $0<p_0\le p\le1$, taking the supremum over the test function~$\varphi$, recognizing the definition~\eqref{eq:def-ANm} of $A_{p_0,q}^{N,m}(t)$, noting that for $p_0\le p$ we have $A_{p_0,q}^{N,m}\le A_{p,q}^{N,m}$, and also noting that the Sobolev embedding implies
\begin{equation}\label{eq:Sob_embedding}
\|\langle z\rangle^{-p_0\ell}\varphi^\ell\|_{W^{r,q'}(\Xd)}\,\lesssim\,\|\langle z\rangle^{-p_0}\varphi\|_{W^{r,q'}(\Xd)}^\ell,\qquad \text{for~$\ell\ge1$,}
\end{equation}
we deduce for all $2\le m\le N$ and $t\ge0$,
\begin{equation*}
\|G^{N,m}_t\|_{W^{-r,q}_{p_0,m}}\,\lesssim_m\,A_{p,q}^{N,m}(t)+\sum_{1\le k\le l\le m-1}\sum_{j_1,\ldots,j_k\ge1\atop j_1+\ldots+j_k=l}N^{l-k-m+1}\prod_{i=1}^k\|G_t^{N,j_i}\|_{W^{-r,q}_{mp_0,j_i}}.
\end{equation*}
Hence, by Proposition~\ref{prop:ito}, for suitable $\lambda_0,p,q,\kappa$,
\begin{equation*}
\|G^{N,m}_t\|_{W^{-r,q}_{p_0,m}}\,\lesssim_m\,N^{-\frac m2}+e^{-\lambda_0pmt}A_{3p,q}^{N,m}(0)
+\sum_{1\le k\le l\le m-1}\sum_{j_1,\ldots,j_k\ge1\atop j_1+\ldots+j_k=l}N^{l-k-m+1}\prod_{i=1}^k\|G_t^{N,j_i}\|_{W^{-r,q}_{mp_0,j_i}}.
\end{equation*}
It remains to express $A_{3p,q}^{N,m}(0)$ in the right-hand side in terms of initial correlations.
For that purpose, we first appeal to the cluster expansion of moments in terms of cumulants: given $\varphi\in C^\infty_c(\Xd)$, by definition of cumulants, we can expand for all~$m\ge1$,
\[\E\bigg[\Big(\int_\Xd\varphi(\mu^N_0-\mu_\circ)\Big)^m\bigg]\,=\,\sum_{\pi\vdash\llbracket m\rrbracket}\prod_{B\in\pi}\kappa^{\sharp B}\Big[\int_\Xd\varphi(\mu^N_0-\mu_\circ)\Big].\]
With the choice $\mu_\circ=F^{N,1}_\circ$, we note that $\kappa^1[\int_\Xd(\mu_0^N-\mu_\circ)]=\E[\int_{\Xd}(\mu_0^N-\mu_\circ)]=0$, and we may thus restrict the sum to partitions $\pi$ such that $\sharp B>1$ for all $B\in\pi$.
Appealing to~\eqref{eq:link-cum-G} at $t=0$ in order to control cumulants back in terms of correlation functions, and taking the supremum over $\varphi$, we deduce
\begin{multline*}
\sup\bigg\{\E \bigg[ \Big( \int_{\Xd} \varphi(\mu^N_0 - \mu_\circ) \Big)^m \bigg]~:~\varphi\in C^\infty_c(\Xd),~\|\langle z\rangle^{-3p}\varphi\|_{W^{2,q'}(\Xd)}\le1\bigg\}\\ 
\,\lesssim\,\sum_{\ell =1}^m \sum_{j_1, \ldots, j_\ell \ge 2\atop j_1 + \ldots + j_\ell = m}\prod_{i=1}^\ell\bigg(\sum_{1 \le k \le l \le j_i} \sum_{n_1, \ldots, n_{k} \ge 1 \atop n_1 + \ldots n_{k} = l} N^{l - k - j_i +1} \prod_{s=1}^{k} \|G^{N,n_s}_\circ\|_{W^{-2,q}_{3mp, n_s}}\bigg).
\end{multline*}
In the right-hand side, we note that the sum over $j_1,\ldots,j_\ell\ge2$ with $j_1+\ldots+j_\ell=m$ requires $2\ell\le m$, so the first sum can be restricted accordingly. Reorganizing the sums, we are then led to
\begin{multline*}
\sup\bigg\{\E \bigg[ \Big( \int_{\Xd} \varphi(\mu^N_0 - \mu_\circ) \Big)^m \bigg]~:~\varphi\in C^\infty_c(\Xd),~\|\langle z\rangle^{-3p}\varphi\|_{W^{2,q'}(\Xd)}\le1\bigg\}\\[-1mm]
\,\lesssim\,\sum_{\ell=1}^{m/2}\sum_{\ell\le K\le L\le m}\sum_{j_1, \ldots, j_{K} \ge 1 \atop j_1 + \ldots + j_{K} = L} N^{L-K-m+\ell} \prod_{i=1}^{K} \|G^{N,j_i}_\circ\|_{W^{-2,q}_{3mp, j_i}}.
\end{multline*}
For $m$ even, this means
\begin{equation*}
A_{3p,q}^{N,m}(0)\,\lesssim\,\sum_{\ell=1}^{m/2}\sum_{\ell\le K\le L\le m}\sum_{j_1, \ldots, j_{K} \ge 1 \atop j_1 + \ldots j_{K} = L} N^{L-K-m+\ell} \prod_{i=1}^{K} \|G^{N,j_i}_\circ\|_{W^{-2,q}_{3mp, j_i}},
\end{equation*}
and for $m$ odd, after combination with~\eqref{eq:estim-ANm-mom0},
\begin{equation*}
A_{3p,q}^{N,m}(0)\,\lesssim\,\sum_{\ell=1}^m\sum_{\ell\le K\le L\le2m}\sum_{1\le j_1,\ldots,j_K\le m+1\atop j_1+\ldots+j_K=L}\bigg(N^{L-K-2m+\ell}\prod_{i=1}^K\|G_\circ^{N,j_i}\|_{W^{-2,q}_{3mp,j_i}}\bigg)^{\frac12}.
\end{equation*}
This concludes the proof.
\end{proof}

\subsection{BBGKY hierarchical estimates}
Starting from the BBGKY hierarchy for marginals of $F^N$, and recalling how marginals can be expanded in terms of correlations (and vice versa), we can derive a corresponding BBGKY hierarchy of equations for correlations.
Before stating it, we start by introducing some useful notation.

\begin{defin0}
Consider a collection $\{h^m\}_{1\le m\le N}$ of functions $h^m:\Xd^m\to\R$ such that for all $m$ the function $h^m$ is symmetric in its $m$ entries (such as $\{F^{N,m}\}_{1\le m\le N}$ or $\{G^{N,m}\}_{1\le m\le N}$).
\begin{enumerate}[---]
\item For $P\subset\llbracket m\rrbracket$ with $P\ne\varnothing$, we define $h^P:\Xd^m\to\R$ as
\[h^P(z_{\llbracket m\rrbracket})\,:=\,h^{\sharp P}(z_P),\]
and for $P=\varnothing$ we set $h^\varnothing:=0$.
\smallskip\item For $P\subset\llbracket m\rrbracket$ with $P\ne\llbracket N\rrbracket$, we define $h^{P\cup\{*\}}:\Xd^{m+1}\to\R$ as
\[h^{P\cup\{*\}}(z_{\llbracket m\rrbracket},z_*)\,:=\,h^{\sharp P+1}(z_P,z_*),\]
and for $P=\llbracket N\rrbracket$ we set $h^{\llbracket N\rrbracket\cup\{*\}}=0$.
\smallskip\item For $P\subset\llbracket m\rrbracket$ and $k,\ell\in P$, we define 
\[S_{k,\ell}h^P\,:=\,\nabla W(x_k - x_\ell) \cdot \nabla_{v_k} h^P.\]
\item For $P\subset\llbracket m\rrbracket$ and $k\in P$, we define
\[H_kh^{P\cup\{*\}}(z_{\llbracket m\rrbracket})\,:=\,\int_\Xd\nabla W(x_k-x_*)\cdot\nabla_{v_k}h^{P\cup\{*\}}(z_{\llbracket m\rrbracket},z_*)\,\ddr x_*.\]
\end{enumerate}
\end{defin0}

With this notation, we may now formulate the BBGKY hierarchy of equations satisfied by correlation functions. To the best of our knowledge, such hierarchies for correlations were first written down by Ernst and Cohen~\cite{Ernst_Cohen_1981} in the context of Boltzmann-type systems. 
We refer to~\cite[Section~4]{Hess_Childs_2023} for the derivation in the case of the overdamped dynamics~\eqref{eq:Brownian_system}: lengthy but straightforward algebraic manipulations are needed to collect the different factors in the correlation hierarchy. 
For convenience, we shall write $A-B=A\setminus B$ for set difference, with for instance $A-B-C=A\setminus(B\cup C)$ and $A\cup B-C=(A\cup B)\setminus C$.

\begin{lem}[Correlation hierarchy, e.g.~\cite{Hess_Childs_2023}]\label{lem:hier-corr}
For fixed $N$, correlation functions $\{G^{N,m}\}_{1 \le m\le N}$ satisfy the following hierarchy of equations: for all $1\le m\le N$,
\begin{align*}
\partial_t G^{N,m} =&\, L_{N,m} G^{N,m} + \kappa \frac{N-m}{N} \sum_{k=1}^m H_k G^{N,\llbracket m \rrbracket \cup \{\ast\}}\\
&- \kappa \sum_{k=1}^m \sum_{A \subset \llbracket m \rrbracket - \{k\}} \frac{m-1-\sharp A}{N} H_k\big(G^{N,A \cup \{k, \ast\}} G^{N,\llbracket m \rrbracket - \{k\} - A} \big)\\
&+ \kappa \frac{N-m}{N} \sum_{k=1}^m \sum_{A \subsetneq \llbracket m \rrbracket - \{k\}\atop A\ne\varnothing} H_k\big( G^{N,A \cup \{k\}} G^{N,\llbracket m \rrbracket \cup \{\ast\} - A - \{k\}} \big) \\
&- \kappa \sum_{k=1}^m \sum_{A \subset \llbracket m\rrbracket - \{k\}} \sum_{B\subset \llbracket m\rrbracket - \{k\} - A} \frac{m-1-\sharp A-\sharp B}{N} H_k\big( G^{N,A \cup \{k\}} G^{N,B \cup \{*\}} G^{N,\llbracket m\rrbracket - A-B - \{k\}} \big) \\
&+ \frac{\kappa}{N} \sum_{k, \ell = 1}^m S_{k,\ell} G^{N,\llbracket m \rrbracket} + \frac{\kappa}{N} \sum_{k \ne \ell}^m \sum_{A \subset \llbracket m\rrbracket - \{k , \ell \}} S_{k, \ell}\big( G^{N,A \cup \{k\}} G^{N,\llbracket m\rrbracket - A - \{k\}}\big) \\
&+ \kappa \frac{N-m}{m} \sum_{k=1}^m H_k \Big( (F^{N,\{*\}}-\mu(z_*))\, G^{N,\llbracket m \rrbracket} + (F^{N,\{k\}}-\mu(z_k))\, G^{N, \llbracket m \rrbracket \cup \{*\}-\{k\}} \Big),
\end{align*}
where $L_{N,m}$ stands for the following (non-autonomous) $m$-particle linearized mean-field operator,
\begin{eqnarray}
L_{N,m}&:=&\sum_{k=1}^m\,\Id^{\otimes k-1}\otimes L_{N,m}^\circ\otimes\Id^{\otimes m-k},\label{eq:def-LNm}\\
L_{N,m}^\circ h&:=&\tfrac12\Div_v((\nabla_v+\beta v)h)-v\cdot\nabla_xh+\Big(\nabla A+\kappa\frac{N-m}N\nabla W\ast \mu\Big)\cdot\nabla_vh\nonumber\\
&&\hspace{5cm}+\kappa\frac{N-m}N\nabla_{v}\mu\cdot\int_{\Xd}\nabla W(x-x_*)h(z_*)\,\ddr z_*.\nonumber
\end{eqnarray}
\end{lem}

Starting from this hierarchy of equations for correlations, combined with ergodic estimates for the linearized mean-field evolution, we establish the following a priori estimates.

\begin{prop}\label{prop:advance_control_GNm}
There is $\lambda_0>0$ (only depending on $d,\beta,a$, $\|W\|_{W^{1,\infty}(\R^d)}$) such that the following holds: given $1<q\le2$ and $0<p\le \frac16$ with $pq'\gg1$ large enough (only depending on $d,\beta,a$), provided that $0<\kappa\ll1$ is small enough (further depending on $p,q$, $Q_2(F_\circ^{N,1})$, and on $\|W\|_{W^{d+3,\infty}\cap H^s(\R^d)}$ for some $s>\frac d2+5$), we have for all $2\le m\le N$, $r \geq 4$, and $t\ge0$,
\begin{multline}
\label{eq:prop_advance_control_G}
\|G^{N,m}_t\|_{W^{-r,q}_{p,m}} \,\lesssim\,e^{-\lambda_0p mt}\|G^{N,m}_\circ\|_{W^{-r,q}_{p,m}}
+\int_0^te^{-\lambda_0 pm(t-s)}\bigg(\|G^{N,m+1}_s\|_{W^{1-r,q}_{p,m+1}}\\
+\sum_{i_1+i_2=m+1\atop 1\le i_1,i_2\le m-1}\|G_s^{N,i_1}\|_{W^{1-r,q}_{p,i_1}} \|G_s^{N,i_2}\|_{W^{1-r,q}_{p,i_2}}
+N^{-1}\!\!\!\sum_{i_1+i_2=m\atop 1\le i_1,i_2\le m-1}\|G^{N,i_1}_s\|_{W^{1-r,q}_{p,i_1}}\|G^{N,i_2}_s\|_{W^{1-r,q}_{p,i_2}}\\
+N^{-1}\!\!\! \sum_{i_1+i_2+i_3=m+1\atop 1\le i_1,i_2,i_3\le m-1} \|G^{N,i_1}_s\|_{W^{1-r,q}_{p,i_1}} \|G^{N,i_2}_s\|_{W^{1-r,q}_{p,i_2}} \|G^{N,i_3}_s\|_{W^{1-r,q}_{p,i_3}} \\[-2mm]
+\Big( N^{-1} + e^{-\lambda_0 p s} \|G^{N,2}_\circ\|_{W^{-2,q}_{3p,2}} \Big)\|G_s^{N,m}\|_{W^{1-r,q}_{p,m}}\bigg)\,\ddr s,
\end{multline}
where the multiplicative constant further depends on $m,r$, $Q_{mp}(F_\circ^{N,1})$, and on $\|W\|_{W^{d+r+1,\infty}\cap H^s(\R^d)}$ for some $s >2r+\tfrac{d}{2}-1$.
\end{prop}

\begin{proof}
We start by recalling the available ergodic estimates for the linearized mean-field operator~$L_{N,m}$.
Given $m\ge1$ and $u_\circ^m\in C^\infty_c(\Xd^m)$ with $\int_\Xd u_\circ^m(z_{\llbracket m\rrbracket})\,dz_j=0$ for all $1\le j\le m$, let us consider the solution~$u^m$ of
\[\left\{\begin{array}{l}
\partial_tu^m=L_{N,m}u^m,\quad t\ge0,\\
u^m|_{t=0}=u^m_\circ,
\end{array}\right.\]
where in the definition~\eqref{eq:def-LNm} of $L_{N,m}$ we recall that we have chosen $\mu$ as the solution of the mean-field equation~\eqref{eq:VFP} with initial condition $\mu_\circ=F^{N,1}_\circ$.
By definition of $L_{N,m}$ as a Kronecker sum, we can write
\[u^m_t\,=\,(V^{N,m}_{0,t})^{\otimes m}u^m_\circ,\]
where $\{V^{N,m}_{s,t}\}_{0\le s\le t}$ stands for the one-particle evolution defined by
\[\left\{\begin{array}{ll}
\partial_tV^{N,m}_{s,t}h=L_{N,m}^\circ V^{N,m}_{s,t},&t\ge s,\\
V^{N,m}_{s,t}h|_{t=s}=h.&s
\end{array}\right.\]
In each component, we may then appeal to the ergodic estimate derived in~\cite[Theorem~2.13]{BD_2024}: there exist $\kappa_0,\lambda_0>0$ (only depending on $d,\beta,a,\|W\|_{W^{1,\infty}(\R^d)}$) such that, given $\kappa\in[0,\kappa_0]$, $1<q\le2$, and $0<p\le1$ with $pq'\gg1$ (only depending on $d,\beta,a$), we have for all $r\ge2$ and $t\ge0$,
\[\|u_t^m\|_{W^{-r,q}_{p,m}}\,\lesssim\,e^{-\lambda_0pmt}\|u_\circ^m\|_{W^{-r,q}_{p,m}},\]
where the multiplicative constant only depends on $d,\beta,a,p,q,r,m$, $Q_2(F^{N,1}_\circ)$, $\|W\|_{W^{d+r+1,\infty}(\R^d)}$.
Testing the definition of $S_{k,\ell}h^P$ with a tensorized test function, decomposing $W$ as a superposition of Fourier modes and using polarization similarly as in~\eqref{eq:bound_fourier_G_N_2} to rewrite the resulting expression as a superposition of integrals of $h^P$ tested with a tensorized test function (recall the definition~\eqref{eq:definition_W_tilde} of functional spaces), we are led to
\[\|S_{k,\ell}h^{P}\|_{W^{-r,q}_{p,\sharp P}} \,\lesssim\, \|h^{P}\|_{{W^{1-r,q}_{p,\sharp P}}},\]
and similarly,
\[\| H_k h^{P\cup \{\ast\}}\|_{W^{-r,q}_{p,\sharp P}} \,\lesssim\, \|h^{P\cup \{\ast\}}\|_{{W^{1-r,q}_{p,\sharp P+1}}},\]
where the multiplicative constants only depend on $d,r,\sharp P$, and $\|W\|_{H^s(\R^d)}$ for some $s > 2r+\tfrac{d}{2}-1$.
Now applying this to the equation for $G^{N,m}$ in Lemma~\ref{lem:hier-corr}, we get for all $2\le m\le N$, $r\ge2$, and $t\ge0$,
\begin{multline*}
\|G^{N,m}_t\|_{W^{-r,q}_{p,m}} \,\lesssim\,e^{-\lambda_0p mt}\|G^{N,m}_\circ\|_{W^{-r,q}_{p,m}}\\
+\int_0^te^{-\lambda_0 pm(t-s)}\bigg(\|G^{N,m+1}_s\|_{W^{1-r,q}_{p,m+1}}
+N^{-1}\!\!\! \sum_{i_1+i_2=m+1\atop 1\le i_1,i_2\le m}\|G_s^{N,i_1}\|_{W^{1-r,q}_{p,i_1}} \|G_s^{N,i_2}\|_{W^{1-r,q}_{p,i_2}}\\
+\sum_{i_1+i_2=m+1\atop 1\le i_1,i_2\le m-1}\|G_s^{N,i_1}\|_{W^{1-r,q}_{p,i_1}} \|G_s^{N,i_2}\|_{W^{1-r,q}_{p,i_2}}
+N^{-1}\!\!\!\sum_{i_1+i_2=m\atop 1\le i_1,i_2\le m-1}\|G^{N,i_1}_s\|_{W^{1-r,q}_{p,i_1}}\|G^{N,i_2}_s\|_{W^{1-r,q}_{p,i_2}}\\
+N^{-1}\!\!\! \sum_{i_1+i_2+i_3=m+1\atop 1\le i_1,i_2,i_3\le m-1} \|G^{N,i_1}_s\|_{W^{1-r,q}_{p,i_1}} \|G^{N,i_2}_s\|_{W^{1-r,q}_{p,i_2}} \|G^{N,i_3}_s\|_{W^{1-r,q}_{p,i_3}}\\[-3mm]
+\Big(N^{-1}+\|F^{N,1}_s-\mu_s\|_{W^{1-r,q}_{p,1}}\Big)\|G^{N,m}_s\|_{W^{1-r,q}_{p,m}}\bigg)\,\ddr s.
\end{multline*}
Note that in the first sum in the right-hand side the terms with $(i_1,i_2)=(m,1)$ or $(1,m)$ can be bounded by $N^{-1}\|G_s^{N,m}\|_{W^{1-r,q}_{p,m}}$, which already appears in the last term. Hence, this first sum can be restricted to $1\le i_1,i_2\le m-1$, which is then bounded by the second sum.
For the last right-hand side term, involving $F^{N,1}-\mu$, we appeal to Theorem~\ref{th:creation} with initial condition $\mu_\circ=F_\circ^{N,1}$:
provided that~$\lambda_0$ is chosen small enough (only depending on $d,\beta,a$, $\|W\|_{W^{1,\infty}(\R^d)}$), for all $1<q\le2$ and $0<p \le \frac16$ with $pq'\gg1$ large enough (only depending on $d,\beta,a$), provided that $\kappa\ll1$ is small enough (further depending on $p,q$, $Q_2(F_\circ^{N,1})$, $\|W\|_{W^{d+3,\infty}\cap H^s(\R^d)}$ for some $s>\frac d2+5$), we have
\[\|F^{N,1}_t-\mu_t\|_{W^{-3,q}_{p,1}}
\,\lesssim\,
N^{-1}+e^{-2\lambda_0pt}\|G^{N,2}_\circ\|_{W^{-2,q}_{3p,2}},\]
which then yields the conclusion.
\end{proof}

\subsection{Proof of Theorem \ref{th:estim-GNm}}\label{subsec:pf_thm_GNm}
The claim~\eqref{eq:estim-GN2-cr} already follows from~\eqref{eq:step1suboptimal} in the proof of Theorem~\ref{th:creation} with the choice $\mu_\circ=F^{N,1}_\circ$.
It remains to prove~\eqref{eq:conclusion_thm_estim_GNm}. Let $1<q\le2$ and $0<p\le\frac16$ with $pq'\gg1$ large enough (only depending on $d,\beta,a$),
let $2 \le m \le N$ and $\alpha\in[0,1]$ be fixed,
and assume that initially
\begin{equation}\label{eq:ass-init}
\|G_\circ^{N,j}\|_{W^{-2,q}_{3p,j}}\,\le\,C_j^\circ N^{-\alpha(j-1)},\qquad\text{for all $1\le j\le 2m - 1$.}
\end{equation}
We split the proof into two main steps.

\medskip\noindent
{\bf Step~1:} Suboptimal estimates.\\
Let $\lambda_0,\kappa$ be as in Proposition~\ref{prop:ito}.
By a straightforward computation, under the initial assumption~\eqref{eq:ass-init}, the result of Proposition~\ref{prop:basic_control_Gm} with $p_0=p/n!$ yields for all $2\le n\le 2m-1$, $r \ge2$, and $t\ge0$,
\begin{equation*}
\|G^{N,n}_t\|_{W^{-r,q}_{p/n!,n}}\,\lesssim\,N^{-\frac n2}
+e^{-\frac{\lambda_0 p}{n!}t}N^{-\alpha\frac n2}
+\sum_{1\le k\le l\le n-1}\sum_{j_1,\ldots,j_k\ge1\atop j_1+\ldots+j_k=l}N^{l-k-n+1}\prod_{i=1}^k\|G_t^{N,j_i}\|_{W^{-r,q}_{p/(n-1)!,j_i}}.
\end{equation*}
By a direct iteration of these  estimates, inductively eliminating correlations in the right-hand side, and using $\|G^{N,1}_t\|{}_{W^{-r,q}_{p,1}}\lesssim1$, we deduce for all $2\le n\le 2m-1$, $r\ge2$, and $t\ge0$,
\begin{equation}\label{eq:suboptimal_proof_Thm_correl}
\|G^{N,n}_t\|_{W^{-r,q}_{p/n!,n}}\,\lesssim\,N^{-\frac n2}
+e^{-\frac{\lambda_0 p}{n!}t}N^{-\alpha\frac n2}.
\end{equation}

\medskip\noindent
{\bf Step~2:} Conclusion via BBGKY analysis.\\
We leverage the suboptimal correlation estimates of Step~1, using the BBGKY estimates of Proposition~\ref{prop:advance_control_GNm} and arguing by induction.
As the desired estimates~\eqref{eq:conclusion_thm_estim_GNm} are already known to hold for two-particle correlations, cf.~\eqref{eq:estim-GN2-cr}, we can assume that there is some $3\le n\le m$ such that the following holds: provided that $\kappa$ is small enough (also depending on $n$), we have for all $t\ge0$,
\begin{equation}\label{eq:induc-hyp}
\|G_t^{N,j}\|_{W_{p_j,j}^{-r_j,q_j}}\,\lesssim\,N^{1-j}+e^{-\lambda_0p_jt}N^{\alpha(1-j)},\qquad\text{for all $2\le j\le n-1$,}
\end{equation}
where we have set
\[r_j:=j+1,\qquad p_j\,:=\,\tfrac{p}{(2j-2)!},\qquad q_j\,:=\,\tfrac{(2j-2)!q'}{(2j-2)!q'-1}.\]
Under this induction assumption, we shall show that the same estimate~\eqref{eq:induc-hyp} is also automatically valid for $j=n$ up to further restricting the range of~$\kappa$, which then concludes the proof.
We split the argument into two further substeps.

\medskip\noindent
{\bf Substep~2.1:} Proof that for $\kappa$ small enough (also depending on $n$) we have for all $0\le k\le n-2$, $r\ge4$, and $t\ge0$,
\begin{multline}\label{eq:plop}
\|G^{N,n+k}_t\|_{W^{-r,q_n}_{p_n,n+k}} \,\lesssim\,
N^{1-n-k}+{e^{-\lambda_0p_nt}N^{\alpha(1-n-k)}}\\
+\int_0^t{e^{-\lambda_0p_n(n+k)(t-s)}}\bigg(\|G^{N,n+k+1}_s\|_{W^{1-r,q_n}_{p_n,n+k+1}}
+\Big( N^{-1} + {e^{-\lambda_0 p_n s}}N^{-\alpha}\Big)\|G_s^{N,n+k}\|_{W^{1-r, q_n}_{p_n,n+k}}\\
+\sum_{i=2}^{k+1}\Big(N^{1-i}+e^{-\lambda_0p_ns}N^{\alpha(1-i)}\Big)\|G_s^{N,n+k+1-i}\|_{W^{1-r, q_n}_{p_n,n+k+1-i}}\bigg)\,\ddr s.
\end{multline}
We apply Proposition~\ref{prop:advance_control_GNm} with $m=n+k$ and with exponents $p_n,q_n$ satisfying $p_nq_n'=pq'\gg1$. Provided that $\kappa$ is small enough (depending on $n$), we get from Proposition~\ref{prop:advance_control_GNm}, for all $0 \le k \le n-2$, $r\ge4$, and $t\ge0$,
\begin{multline*}
\|G^{N,n+k}_t\|_{W^{-r,q_n}_{p_n,n+k}} \,\lesssim\,e^{-\lambda_0p_nt}\|G^{N,n+k}_\circ\|_{W^{-r,q_n}_{p_n,n+k}}
+\int_0^t{e^{-\lambda_0p_n(n+k)(t-s)}}\bigg(\|G^{N,n+k+1}_s\|_{W^{1-r,q_n}_{p_n,n+k+1}}\\
+\sum_{i_1+i_2=n+k+1\atop 1\le i_1,i_2\le n+k-1}\|G_s^{N,i_1}\|_{W^{1-r,q_n}_{p_n,i_1}} \|G_s^{N,i_2}\|_{W^{1-r,q_n}_{p_n,i_2}}
+N^{-1}\!\!\!\sum_{i_1+i_2=n+k\atop 1\le i_1,i_2\le n+k-1}\|G^{N,i_1}_s\|_{W^{1-r,q_n}_{p_n,i_1}}\|G^{N,i_2}_s\|_{W^{1-r,q_n}_{p_n,i_2}}\\
+N^{-1}\!\!\! \sum_{i_1+i_2+i_3=n+k+1\atop 1\le i_1,i_2,i_3\le n+k-1} \|G^{N,i_1}_s\|_{W^{1-r,q_n}_{p_n,i_1}} \|G^{N,i_2}_s\|_{W^{1-r,q_n}_{p_n,i_2}} \|G^{N,i_3}_s\|_{W^{1-r,q_n}_{p_n,i_3}} \\[-2mm]
+\Big( N^{-1} + e^{-\lambda_0 p_n s} \|G^{N,2}_\circ\|_{W^{-2,q_n}_{3p_n,2}} \Big)\|G_s^{N,n+k}\|_{W^{1-r,q_n}_{p_n,n+k}}\bigg)\,\ddr s.
\end{multline*}
Note that the conditions $p_jq_j'=pq'>2d$ and $r_j=j+1$ entail by Jensen's inequality, for $j\le n-1$,
\[\|G^{N,j}\|_{W^{1-r_n,q_n}_{p_n,j}}\,\lesssim\,\|G^{N,j}\|_{W^{-r_j,q_j}_{p_j,j}}.\]
Using the initial assumption~\eqref{eq:ass-init}, as well as the induction assumption~\eqref{eq:induc-hyp} for $G^{N,j}$ with $j\le n-1$, the claim~\eqref{eq:plop} follows from the above estimate after straightforward simplifications.

\medskip\noindent
{\bf Substep~2.2:} Proof that for $\kappa$ small enough (depending on $n$) we have for all $0 \le \ell \le n-2$ and~$t\ge0$,
\begin{equation}\label{eq:claim_correl}
\|G^{N,n+k}_t\|_{W^{-3-\ell,q_n}_{p_n,n+k}} \,\lesssim\, N^{-\frac{n+k+\ell}{2}} +e^{-\lambda_0p_nt} N^{-\alpha \frac{n+k+\ell}{2}}, \qquad\text{for all $0 \le k \le n-2-\ell$}.
\end{equation}
This will conclude the proof: indeed, choosing $\ell=n-2$ and $k=0$, the above estimate precisely yields~\eqref{eq:induc-hyp} with $j=n$ as desired.
We turn to the proof of~\eqref{eq:claim_correl} and argue by induction on $\ell$. As for $\ell=0$ the claim~\eqref{eq:claim_correl} follows from~\eqref{eq:suboptimal_proof_Thm_correl}, we can assume that there is some $1\le L< n-1$ such that~\eqref{eq:claim_correl} holds with $\ell=L-1$, that is,
\begin{equation}\label{eq:claim_correl-induc}
\|G^{N,n+k}_t\|_{W^{-2-L,q_n}_{p_n,n+k}} \,\lesssim\, N^{-\frac{n+k+L-1}{2}} + e^{-\lambda_0 p_n t} N^{-\alpha \frac{n+k+L-1}{2}},\qquad\text{for all $0\le k\le n-1-L$},
\end{equation}
and it remains to show that~\eqref{eq:claim_correl} then automatically also holds for $\ell=L$.
This immediately follows from the result~\eqref{eq:plop} of Step~2.1 combined with the induction assumption~\eqref{eq:claim_correl-induc}.
\qed

\section*{Acknowledgements}
The authors thank Arnaud Guillin for pointing out the validity of the creation-of-chaos estimate~\eqref{eq:creation_strong} in Wasserstein distance as a consequence of~\cite{Bolley_2010}.
AB acknowledges financial support from the AAP ``Accueil'' from Universit\'e Lyon 1 Claude Bernard.
MD and MM acknowledge financial support from the European Union (ERC, PASTIS, Grant Agreement n$^\circ$101075879). Views and opinions expressed are however those of the authors only and do not necessarily reflect those of the European Union or the European Research Council Executive Agency. Neither the European Union nor the granting authority can be held responsible for them.

\bibliographystyle{plain}
\bibliography{biblio}

\end{document}